\documentclass{amsart}
\usepackage[top=1in,bottom=1in,left=1.1in,right=1.1in]{geometry}
\usepackage{graphicx}
\usepackage{float}
\usepackage{amsthm}
\usepackage{amsmath}
\usepackage{amssymb}
\usepackage{mathrsfs}
\usepackage{tikz-cd}
\tikzcdset{scale cd/.style={every label/.append style={scale=#1},
    cells={nodes={scale=#1}}}}
\usepackage{adjustbox}
\usepackage{relsize}
\usepackage{stmaryrd}
\usepackage{enumitem}
\usepackage{placeins}
\usepackage{pdflscape}
\usepackage[colorlinks=true,citecolor=blue]{hyperref}
\usepackage{appendix}

\usepackage{comment}

\theoremstyle{definition}
\newtheorem{Definition}{Definition}[section]

\theoremstyle{plain}
\newtheorem{Theorem}[Definition]{Theorem}

\theoremstyle{plain}
\newtheorem{Proposition}[Definition]{Proposition}

\theoremstyle{definition}
\newtheorem{Notation}[Definition]{Notation}

\theoremstyle{plain}
\newtheorem{Lemma}[Definition]{Lemma}

\theoremstyle{plain}

\theoremstyle{plain}

\theoremstyle{plain}
\newtheorem{Construction}[Definition]{Construction}

\theoremstyle{definition}
\newtheorem{Example}[Definition]{Example}

\theoremstyle{remark}
\newtheorem{Remark}[Definition]{Remark}

\theoremstyle{plain}
\newcommand{\thistheoremname}{}
\newtheorem*{generic*}{\thistheoremname}
\newenvironment{namedthm*}[1]
  {\renewcommand{\thistheoremname}{#1}%
   \begin{generic*}}
  {\end{generic*}}

\newcommand{\stringdiagramfigurescale}{0.62}
\newcommand{\stringdiagramfigure}[1]{\adjustbox{valign=c}{\includegraphics[scale=\stringdiagramfigurescale]{#1}}}

\title{Frobenius Algebras and Dual Bimodules in Monoidal 2-Categories}
\author{Hao Xu}
\date{\today}

\begin{document}

\bibliographystyle{alpha}

\begin{abstract}
    We explicitly construct dual bimodules in a semistrict monoidal 2-category, using Frobenius algebra structure. The main result shows that a coherent dual of the underlying object can be promoted to a coherent dual of the bimodule, with zigzag 2-isomorphisms additionally require special Frobenius structures. We also prove that every special Frobenius algebra in $\mathbf{2Vect}$ is rigid, via a categorified Casimir object argument, and discuss the relationship between the Frobenius, rigid, special Frobenius, and separable algebra hierarchies.
\end{abstract}

\maketitle

{\hypersetup{linkcolor=black}\tableofcontents}

\section{Introduction}

Dualizable objects and morphisms play a central role in higher category theory, with applications to topological field theory, representation theory, and the classification of fusion 2-categories. A classical result in monoidal 1-categories states that (i) if bimodules over Frobenius algebras are dualizable, then the forgetful functor reflects dualizability; (ii) if the underlying objects are dualizable, then the duality data can be lifted to the bimodules if the Frobenius algebra is special. The purpose of this paper is to establish the analogous result in monoidal 2-categories.

We work in a semistrict monoidal 2-category $\mathfrak{C}$ and consider $(A,B)$-bimodules $M$ where $A$ and $B$ are Frobenius algebras. Our main result provides an explicit construction of a coherent right dual for $M$: given a coherent right dual $M^\lor$ of the underlying object in $\mathfrak{C}$, we equip $M^\lor$ with a $(B,A)$-bimodule structure and construct the unit and counit 1-morphisms and their zigzag 2-isomorphisms. The bimodule structure and unit/counit only require Frobenius structure on $A$; the zigzag 2-isomorphisms additionally require the special Frobenius section $\gamma^A$. Similarly, given a coherent left dual ${}^\lor M$ of the underlying object, the bimodule structure and unit/counit only require Frobenius structure on $B$, while the zigzag 2-isomorphisms require the special Frobenius section $\gamma^B$.

This construction is more general than Décoppet's proof in the literature via 2-adjunctions \cite{D8}, which guarantees dualizability of bimodules over separable algebras in fusion 2-categories but without explicit formulas. We work in a general semistrict monoidal 2-category and need only Frobenius and special Frobenius structures rather than full separability.

We also prove that every special Frobenius algebra in $\mathbf{2Vect}$ is rigid (Proposition~\ref{prop:Frobenius_algebra_in_2Vect_is_rigid}), via a categorified Casimir object argument, and raise the open question of whether this holds in a general fusion 2-category.

\subsection*{Acknowledgements}
    The author was supported by DAAD Graduate School Scholarship Programme (57572629), DFG Project 398436923 RTG 2491 ``\textit{Fourier Analysis and Spectral Theory}'' and Villum Fonden 00060714 ``\textit{Global Categorical Symmetries and Phases of Quantum Matter}''. The author would like to thank Thibault Décoppet and Markus Zetto for helpful discussions.
\section{Preliminaries}

\subsection{Graphical Calculus}

We use the string diagram calculus for semistrict monoidal 2-categories throughout. Objects correspond to regions, 1-morphisms to strings, and 2-morphisms to coupons (vertices). Horizontal composition corresponds to horizontal juxtaposition of string diagrams \emph{from left to right}, and vertical composition to stacking \emph{from top to bottom}. We follow the conventions of \cite{DX,D8}; the reader unfamiliar with string diagrams for 2-categories may consult these references for the relevant notation and coherence results.

\subsection{Coherent Duals}

Suppose $(\mathfrak{C},\Box,\mathbf{I})$ is a semistrict monoidal 2-category.

\begin{Definition} \label{def:CoherentDual}
    A coherent right dual of an object $y$ in $\mathfrak{C}$ consists of:
    \begin{enumerate}
        \item An object $y^\lor$ in $\mathfrak{C}$;
        
        \item A unit 1-morphism $c:\mathbf{I} \to y^\lor \, \Box \, y$ in $\mathfrak{C}$;
        
        \item A counit 1-morphism $e:y \, \Box \, y^\lor \to \mathbf{I}$ in $\mathfrak{C}$;
        
        \item Two 2-isomorphisms \[\begin{tikzcd}
            {y^\lor}
                \arrow[r,"c 1"]
                \arrow[d,equal]
            & { y^\lor \, \Box \, y \, \Box \, y^\lor}
                \arrow[d,"1 e"]
                \arrow[dl,Rightarrow,shorten <= 10pt,shorten >= 10pt,"E"]
            \\ {y^\lor}
                \arrow[r,equal]
            & {y^\lor}
        \end{tikzcd}, \begin{tikzcd}
            {y}
                \arrow[r,"1 c"]
                \arrow[d,equal]
            & { y \, \Box \, y^\lor \, \Box \, y}
                \arrow[d,"e 1"]
                \arrow[dl,Rightarrow,shorten <= 10pt,shorten >= 10pt,"F"]
            \\ {y}
                \arrow[r,equal]
            & {y}
        \end{tikzcd}, \]
    \end{enumerate} satisfying

    \begin{enumerate}
        \item [a.] We have the coherence equation
    \end{enumerate}

    \begin{equation} \label{eqn:CoherentDual1}
        \begin{tabular}{@{}ccc@{}}
        
        \stringdiagramfigure{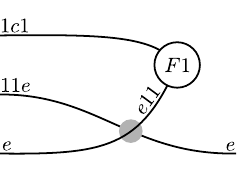} & \adjustbox{valign=c}{$=$} &
        \stringdiagramfigure{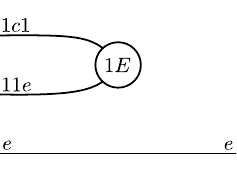}
        
        \end{tabular}
    \end{equation}

    \begin{enumerate}
        \item [] in $\mathbf{Hom}_\mathfrak{C}({}^\lor x \, \Box \, x,\mathbf{I})$,
        
        \item [b.] We have the coherence equation
    \end{enumerate}

    \begin{equation} \label{eqn:CoherentDual2}
        \begin{tabular}{@{}ccc@{}}
        
        \stringdiagramfigure{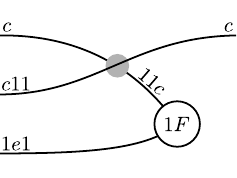} & \adjustbox{valign=c}{$=$} &
        \stringdiagramfigure{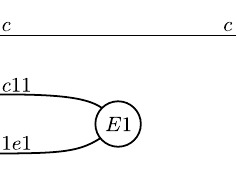}
        
        \end{tabular}
    \end{equation}

    \begin{enumerate}
        \item [] in $\mathbf{Hom}_\mathfrak{C}(\mathbf{I},x \, \Box \, {}^\lor x)$.
    \end{enumerate}

    Dually, a coherent left dual of an object $x$ in $\mathfrak{C}$ consists of:
    \begin{enumerate}
        \item An object ${}^\lor x$ in $\mathfrak{C}$;
        
        \item A unit 1-morphism $c:\mathbf{I} \to x \, \Box \, {}^\lor x$ in $\mathfrak{C}$;
        
        \item A counit 1-morphism $e:{}^\lor x \, \Box \, x \to \mathbf{I}$ in $\mathfrak{C}$;
        
        \item Two 2-isomorphisms \[\begin{tikzcd}
            {x}
                \arrow[r,"c 1"]
                \arrow[d,equal]
            & { x \, \Box \, {}^\lor x \, \Box \, x}
                \arrow[d,"1 e"]
                \arrow[dl,Rightarrow,shorten <= 10pt,shorten >= 10pt,"E"]
            \\ {x}
                \arrow[r,equal]
            & {x}
        \end{tikzcd}, \begin{tikzcd}
            {{}^\lor x}
                \arrow[r,"1 c"]
                \arrow[d,equal]
            & { {}^\lor x \, \Box \, x \, \Box \, {}^\lor x}
                \arrow[d,"e 1"]
                \arrow[dl,Rightarrow,shorten <= 10pt,shorten >= 10pt,"F"]
            \\ {{}^\lor x}
                \arrow[r,equal]
            & {{}^\lor x}
        \end{tikzcd},\]
    \end{enumerate} satisfying the same two coherence equations as above.
\end{Definition}

\subsection{Algebras} \label{sec:Algebras}

Suppose $(\mathfrak{C},\Box,\mathbf{I})$ is a semistrict monoidal 2-category.

\begin{Definition} \label{def:Algebra}
    An algebra in $\mathfrak{C}$ consists of:
    \begin{enumerate}
        \item An object $A$ in $\mathfrak{C}$;
        
        \item Two 1-morphisms, multiplication $m: A \, \Box \, A \to A$ and unit $i: \mathbf{I} \to A$;
        
        \item Three 2-isomorphisms, associator $\alpha$, left unitor $\lambda$ and right unitor $\rho$: \[\begin{tikzcd}
            {A \, \Box \, A \, \Box \, A} 
                \arrow[r, "m 1"]
                \arrow[d,"1 m"']
            & {A \, \Box \, A}
                \arrow[d, "m"]
                \arrow[ld,Rightarrow,shorten <= 10pt,shorten >= 10pt,"\alpha"]
            \\ {A \, \Box \, A}
                \arrow[r, "m"']
            & {A}
        \end{tikzcd}, \begin{tikzcd}
            {\mathbf{I} \, \Box \, A} 
                \arrow[r, "i 1"]
                \arrow[d,equal]
            & {A \, \Box \, A}
                \arrow[d, "m"]
                \arrow[dl,Rightarrow,shorten <= 10pt,shorten >= 10pt,"\lambda"']
                \arrow[dr,Rightarrow,shorten <= 10pt,shorten >= 10pt,"\rho"]
            & {A \, \Box \, \mathbf{I}} 
                \arrow[d,equal]
                \arrow[l,"1 i"']
            \\ {A}
                \arrow[r, equal]
            & {A}
            & {A}
                \arrow[l,equal]
        \end{tikzcd},\]
    \end{enumerate} subject to the following coherence conditions:
    
    \begin{enumerate}
        \item [a.] We have the pentagon equation in $\mathbf{Hom}_\mathfrak{C}(A \, \Box \, A \, \Box \, A \, \Box \, A,A)$:
    \end{enumerate}

    \begin{equation} \label{eqn:AlgebraAssociativity}
        \begin{tabular}{@{}cccc@{}}
        
        \stringdiagramfigure{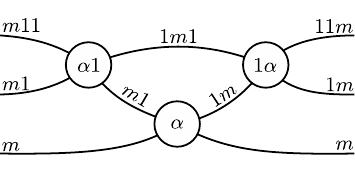} & \adjustbox{valign=c}{$=$} &
        \stringdiagramfigure{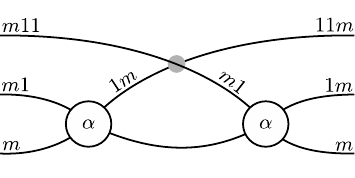} & \adjustbox{valign=c}{,}
        
        \end{tabular}
    \end{equation}

    \begin{enumerate}
        \item [b.] We have the triangle equation in $\mathbf{Hom}_\mathfrak{C}(A \, \Box \, A,A)$:
    \end{enumerate}

    \begin{equation} \label{eqn:AlgebraUnitality}
        \begin{tabular}{@{}cccc@{}}
        
        \stringdiagramfigure{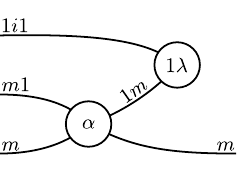} & \adjustbox{valign=c}{$=$} &
        \stringdiagramfigure{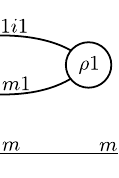} & \adjustbox{valign=c}{.}
        
        \end{tabular}
    \end{equation}
\end{Definition}

\subsection{Modules}

Suppose $(\mathfrak{C},\Box,\mathbf{I})$ is a semistrict monoidal 2-category.

\begin{Definition} \label{def:LeftModule}
    Let $(A,m,i,\alpha,\lambda,\rho)$ be an algebra in $\mathfrak{C}$. A left $A$-module in $\mathfrak{C}$ consists of:
    \begin{enumerate}
        \item An object $M$ in $\mathfrak{C}$;
        
        \item A 1-morphism, left $A$-action $l^M: A \, \Box \, M \to M$;
        
        \item Two 2-isomorphisms, associator $\mu^M$ and left unitor $\lambda^M$: \[\begin{tikzcd}
            {A \, \Box \, A \, \Box \, M} 
                \arrow[r, "m 1"]
                \arrow[d,"1 l^M"']
            & {A \, \Box \, M}
                \arrow[d, "l^M"]
                \arrow[ld,Rightarrow,shorten <= 10pt,shorten >= 10pt,"\mu^M"]
            \\ {A \, \Box \, M}
                \arrow[r, "l^M"']
            & {M}
        \end{tikzcd}, \begin{tikzcd}
            {\mathbf{I} \, \Box \, M} 
                \arrow[r, "i 1"]
                \arrow[d,equal]
            & {A \, \Box \, M}
                \arrow[d, "l^M"]
                \arrow[dl,Rightarrow,shorten <= 10pt,shorten >= 10pt,"\lambda^M"]
            \\ {M}
                \arrow[r,equal]
            & {M}
        \end{tikzcd},\]
    \end{enumerate} subject to the following coherence conditions:

    \begin{enumerate}
        \item [a.] We have the pentagon equation
    \end{enumerate}

    \begin{equation} \label{eqn:LeftModAssociativity}
        \begin{tabular}{@{}ccc@{}}
        
        \stringdiagramfigure{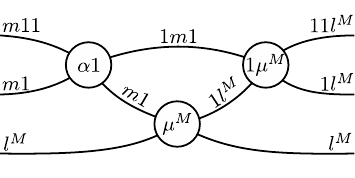} & \adjustbox{valign=c}{$=$} &
        \stringdiagramfigure{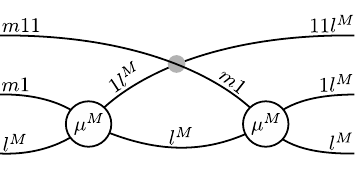}
        
        \end{tabular}
    \end{equation}

    \begin{enumerate}
        \item [] in $\mathbf{Hom}_\mathfrak{C}(A \, \Box \, A \, \Box \, A \, \Box \, M,M)$,
        
        \item [b.] We have the triangle equation
    \end{enumerate}

    \begin{equation} \label{eqn:LeftModUnitality}
        \begin{tabular}{@{}ccc@{}}
        
        \stringdiagramfigure{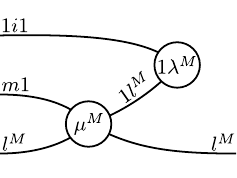} & \adjustbox{valign=c}{$=$} &
        \stringdiagramfigure{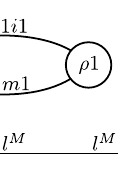}
        
        \end{tabular}
    \end{equation}

    \begin{enumerate}
        \item [] in $\mathbf{Hom}_\mathfrak{C}(A \, \Box \, M,M)$.
    \end{enumerate}
\end{Definition}

\begin{Definition} \label{def:RightModule}
    Dually, a right $A$-module in $\mathfrak{C}$ consists of:
    \begin{enumerate}
        \item An object $N$ in $\mathfrak{C}$;
        
        \item A 1-morphism, right $A$-action $n^N: N \, \Box \, A \to N$;
        
        \item Two 2-isomorphisms, associator $\nu^N$ and right unitor $\rho^N$: \[\begin{tikzcd}
            {N \, \Box \, A \, \Box \, A} 
                \arrow[r, "n^N 1"]
                \arrow[d,"1 m"']
            & {N \, \Box \, A}
                \arrow[d, "n^N"]
                \arrow[ld,Rightarrow,shorten <= 10pt,shorten >= 10pt,"\nu^N"]
            \\ {N \, \Box \, A}
                \arrow[r, "n^N"']
            & {N}
        \end{tikzcd}, \begin{tikzcd}
            {N \, \Box \, \mathbf{I}} 
                \arrow[r, "1 i"]
                \arrow[d,equal]
            & {N \, \Box \, A}
                \arrow[d, "n^N"]
                \arrow[dl,Rightarrow,shorten <= 10pt,shorten >= 10pt,"\rho^N"]
            \\ {N}
                \arrow[r,equal]
            & {N}
        \end{tikzcd},\]
    \end{enumerate} subject to the following coherence conditions:

    \begin{enumerate}
        \item [a.] We have the pentagon equation
    \end{enumerate}

    \begin{equation} \label{eqn:RightModAssociativity}
        \begin{tabular}{@{}ccc@{}}
        
        \stringdiagramfigure{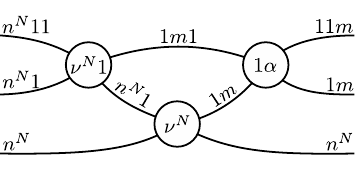} & \adjustbox{valign=c}{$=$} &
        \stringdiagramfigure{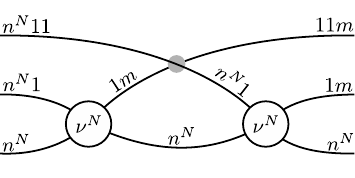}
        
        \end{tabular}
    \end{equation}

    \begin{enumerate}
        \item [] in $\mathbf{Hom}_\mathfrak{C}(N \, \Box \, A \, \Box \, A \, \Box \, A,N)$,
        
        \item [b.] We have the triangle equation
    \end{enumerate}

    \begin{equation} \label{eqn:RightModUnitality}
        \begin{tabular}{@{}ccc@{}}
        
        \stringdiagramfigure{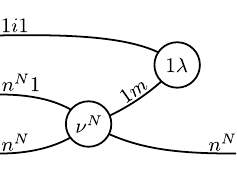} & \adjustbox{valign=c}{$=$} &
        \stringdiagramfigure{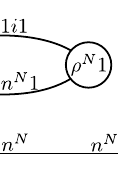}
        
        \end{tabular}
    \end{equation}

    \begin{enumerate}
        \item [] in $\mathbf{Hom}_\mathfrak{C}(N \, \Box \, A,N)$.
    \end{enumerate}
\end{Definition}

\begin{Definition} \label{def:Bimodule}
    Let $(A,m^A,i^A,\alpha^A,\lambda^A,\rho^A)$ and $(B,m^B,i^B,\alpha^B,\lambda^B,\rho^B)$ be two algebras in $\mathfrak{C}$. An $(A,B)$-bimodule in $\mathfrak{C}$ consists of:
    \begin{enumerate}
        \item An underlying object $M$ in $\mathfrak{C}$;
        
        \item A left $A$-module structure on $M$, $(l^M,\mu^M,\lambda^M)$;
        
        \item A right $B$-module structure on $M$, $(n^M,\nu^M,\rho^M)$;
        
        \item An additional 2-isomorphism, associator $\alpha^M$: \[\begin{tikzcd}
            {A \, \Box \, M \, \Box \, B} 
                \arrow[r, "l^M 1"]
                \arrow[d,"1 n^M"']
            & {M \, \Box \, B}
                \arrow[d, "n^M"]
                \arrow[ld,Rightarrow,shorten <= 10pt,shorten >= 10pt,"\alpha^M"]
            \\ {A \, \Box \, M}
                \arrow[r, "l^M"']
            & {M}
        \end{tikzcd}, \]
    \end{enumerate} subject to the following coherence conditions:

    \begin{enumerate}
        \item [a.] We have the pentagon equation
    \end{enumerate}

    \begin{equation} \label{eqn:BiModAssociativity1}
        \begin{tabular}{@{}ccc@{}}
        
        \stringdiagramfigure{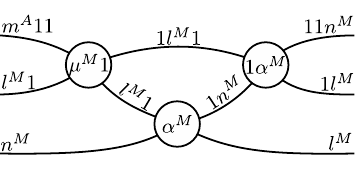} & \adjustbox{valign=c}{$=$} &
        \stringdiagramfigure{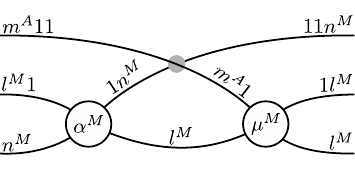}
        
        \end{tabular}
    \end{equation}

    \begin{enumerate}
        \item [] in $\mathbf{Hom}_\mathfrak{C}(A \, \Box \, A \, \Box \, M \, \Box \, B,M)$,
        
        \item [b.] We have the pentagon equation
    \end{enumerate}

    \begin{equation} \label{eqn:BiModAssociativity2}
        \begin{tabular}{@{}ccc@{}}
        
        \stringdiagramfigure{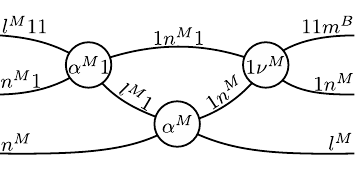} & \adjustbox{valign=c}{$=$} &
        \stringdiagramfigure{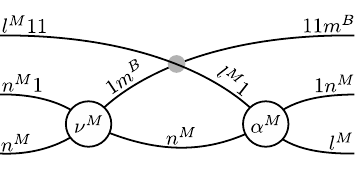}
        
        \end{tabular}
    \end{equation}

    \begin{enumerate}
        \item [] in $\mathbf{Hom}_\mathfrak{C}(A \, \Box \, M \, \Box \, B \, \Box \, B,M)$.
    \end{enumerate}
\end{Definition}

\begin{Definition} \label{def:LeftModule1Morphism}
    Let $M,N$ be two left $A$-modules in $\mathfrak{C}$. A left $A$-module 1-morphism $f: M \to N$ consists of:
    \begin{enumerate}
        \item An underlying 1-morphism $f: M \to N$ in $\mathfrak{C}$;
        
        \item A 2-isomorphism $\psi^f$: \[\begin{tikzcd}
            {A \, \Box \, M}
                \arrow[r,"1 f"]
                \arrow[d,"l^M"']
            & {A \, \Box \, N}
                \arrow[d,"l^N"]
                \arrow[ld,Rightarrow,shorten <= 10pt,shorten >= 10pt,"\psi^f"]
            \\ {M}
                \arrow[r,"f"']
            & {N}
        \end{tikzcd}, \]
    \end{enumerate} subject to the following coherence conditions:

    \begin{enumerate}
        \item [a.] We have the pentagon equation
    \end{enumerate}

    \begin{equation} \label{eqn:LeftMod1MorAssociativity}
        \begin{tabular}{@{}ccc@{}}
        
        \stringdiagramfigure{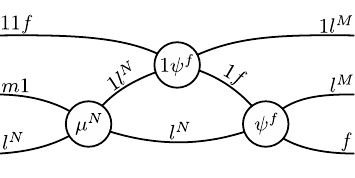} & \adjustbox{valign=c}{$=$} &
        \stringdiagramfigure{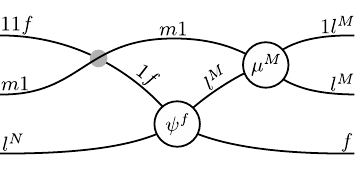}
        
        \end{tabular}
    \end{equation}

    \begin{enumerate}
        \item [] in $\mathbf{Hom}_\mathfrak{C}(A \, \Box \, A \, \Box \, M,N)$,
        
        \item [b.] We have the triangle equation
    \end{enumerate}

    \begin{equation} \label{eqn:LeftMod1MorUnitality}
        \begin{tabular}{@{}ccc@{}}
        
        \stringdiagramfigure{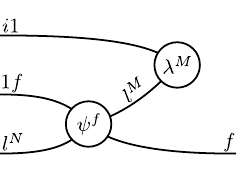} & \adjustbox{valign=c}{$=$} &
        \stringdiagramfigure{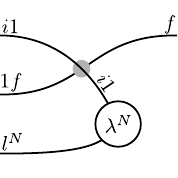}
        
        \end{tabular}
    \end{equation}

    \begin{enumerate}
        \item [] in $\mathbf{Hom}_\mathfrak{C}(M,N)$.
    \end{enumerate}
\end{Definition}

\begin{Definition} \label{def:RightModule1Morphism}
    Let $M,N$ be two right $B$-modules in $\mathfrak{C}$. A right $B$-module 1-morphism $f: M \to N$ consists of:
    \begin{enumerate}
        \item An underlying 1-morphism $f: M \to N$ in $\mathfrak{C}$;
        
        \item A 2-isomorphism $\psi^f$: \[\begin{tikzcd}
            {M \, \Box \, B}
                \arrow[r,"f 1"]
                \arrow[d,"n^M"']
            & {N \, \Box \, B}
                \arrow[d,"n^N"]
                \arrow[ld,Rightarrow,shorten <= 10pt,shorten >= 10pt,"\psi^f"]
            \\ {M}
                \arrow[r,"f"']
            & {N}
        \end{tikzcd}, \]
    \end{enumerate} subject to the following coherence conditions:

    \begin{enumerate}
        \item [a.] We have the pentagon equation
    \end{enumerate}

    \begin{equation} \label{eqn:RightMod1MorAssociativity}
        \begin{tabular}{@{}ccc@{}}
        
        \stringdiagramfigure{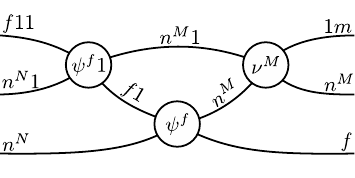} & \adjustbox{valign=c}{$=$} &
        \stringdiagramfigure{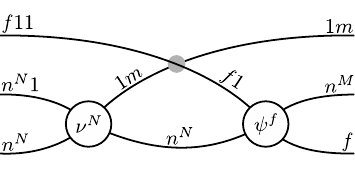}
        
        \end{tabular}
    \end{equation}

    \begin{enumerate}
        \item [] in $\mathbf{Hom}_\mathfrak{C}(M \, \Box \, B \, \Box \, B,N)$,
        
        \item [b.] We have the triangle equation
    \end{enumerate}

    \begin{equation} \label{eqn:RightMod1MorUnitality}
        \begin{tabular}{@{}ccc@{}}
        
        \stringdiagramfigure{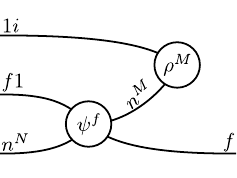} & \adjustbox{valign=c}{$=$} &
        \stringdiagramfigure{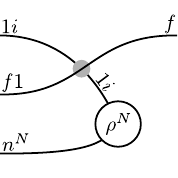}
        
        \end{tabular}
    \end{equation}

    \begin{enumerate}
        \item [] in $\mathbf{Hom}_\mathfrak{C}(M,N)$.
    \end{enumerate}
\end{Definition}

\begin{Definition} \label{def:Bimodule1Morphism}
    Let $M,N$ be two $(A,B)$-bimodules in $\mathfrak{C}$. An $(A,B)$-bimodule 1-morphism $f: M \to N$ consists of a 1-morphism $f: M \to N$ in $\mathfrak{C}$ with a left $A$-module 1-morphism structure $\varphi^f$ and a right $B$-module 1-morphism structure $\psi^f$, such that the hexagon equation

    \begin{equation} \label{eqn:BiMod1Mor}
        \begin{tabular}{@{}ccc@{}}
        
        \stringdiagramfigure{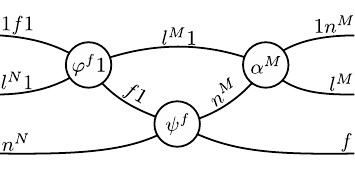} & \adjustbox{valign=c}{$=$} &
        \stringdiagramfigure{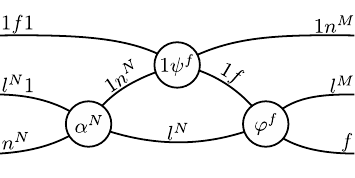}
        
        \end{tabular}
    \end{equation} holds in $\mathbf{Hom}_\mathfrak{C}(A \, \Box \, M \, \Box \, B,N)$.
\end{Definition}

\begin{Definition} \label{def:LeftModule2Morphism}
    Let $M,N$ be two left $A$-modules, and $f,g: M \to N$ be two left $A$-module 1-morphisms in $\mathfrak{C}$. A left $A$-module 2-morphism $\xi$ from $(f,\psi^f)$ to $(g,\psi^g)$ consists of a 2-morphism $\xi: f \to g$ in $\mathfrak{C}$ such that the following coherence condition

    \begin{equation} \label{eqn:LeftMod2Mor}
        \begin{tabular}{@{}ccc@{}}
        
        \stringdiagramfigure{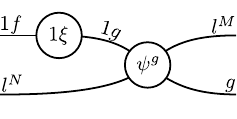} & \adjustbox{valign=c}{$=$} &
        \stringdiagramfigure{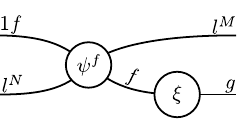}
        
        \end{tabular}
    \end{equation} holds in $\mathbf{Hom}_\mathfrak{C}(A \, \Box \, M,N)$.
\end{Definition}

\begin{Definition} \label{def:RightModule2Morphism}
    Let $M,N$ be two right $B$-modules, and $f,g: M \to N$ be two right $B$-module 1-morphisms in $\mathfrak{C}$. A right $B$-module 2-morphism $\xi$ from $(f,\psi^f)$ to $(g,\psi^g)$ consists of a 2-morphism $\xi: f \to g$ in $\mathfrak{C}$ such that the following coherence condition

    \begin{equation} \label{eqn:RightMod2Mor}
        \begin{tabular}{@{}ccc@{}}
        
        \stringdiagramfigure{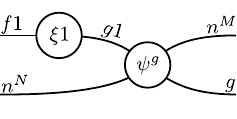} & \adjustbox{valign=c}{$=$} &
        \stringdiagramfigure{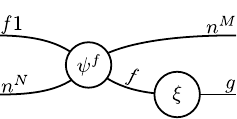}
        
        \end{tabular}
    \end{equation} holds in $\mathbf{Hom}_\mathfrak{C}(M \, \Box \, B,N)$.
\end{Definition}

\begin{Definition} \label{def:Bimodule2Morphism}
    Let $M,N$ be two $(A,B)$-bimodules, and $f,g: M \to N$ be two $(A,B)$-bimodule 1-morphisms in $\mathfrak{C}$. An $(A,B)$-bimodule 2-morphism $\xi$ from $(f,\chi^f,\psi^f)$ to $(g,\chi^g,\psi^g)$ consists of a 2-morphism $\xi: f \to g$ in $\mathfrak{C}$ such that $\xi$ is both a left $A$-module 2-morphism and a right $B$-module 2-morphism.
\end{Definition}

\begin{Proposition}
    For algebras $A,B$ in $\mathfrak{C}$, one has the following 2-categories:
    \begin{enumerate}
        \item There is a 2-category $\mathbf{Lmod}_\mathfrak{C}(A)$ whose \begin{itemize}
            \item objects are left $A$-modules,
            
            \item 1-morphisms are left $A$-module 1-morphisms, 
            
            \item 2-morphisms are left $A$-module 2-morphisms;
        \end{itemize} 

        \item There is a 2-category $\mathbf{Mod}_\mathfrak{C}(B)$ whose \begin{itemize}
            \item objects are right $B$-modules,
            
            \item 1-morphisms are right $B$-module 1-morphisms, 
            
            \item 2-morphisms are right $B$-module 2-morphisms;
        \end{itemize} 

        \item There is a 2-category $\mathbf{Bimod}_\mathfrak{C}(A,B)$ whose \begin{itemize}
            \item objects are $(A,B)$-bimodules,
            
            \item 1-morphisms are $(A,B)$-bimodule 1-morphisms, 
            
            \item 2-morphisms are $(A,B)$-bimodule 2-morphisms.
        \end{itemize}
    \end{enumerate}
\end{Proposition}

\begin{Example}
    In $\mathfrak{C} = \mathbf{2Vect}$, the above notions reproduce the usual notions of left, right and bimodule categories \cite{Mu1,Mu2,Ost1}.
\end{Example}

\subsection{Relative Tensor Product}

Suppose $(\mathfrak{C},\Box,\mathbf{I})$ is a semistrict monoidal 2-category. Let $B$ be an algebra, $M$ be a right $B$-module, $N$ be a left $B$-module, $L$ be an object in $\mathfrak{C}$. 

\begin{Definition} \label{def:Balanced1Morphism}
    A $B$-balanced 1-morphism consists of 
    \begin{enumerate}
        \item A 1-morphism $f: M \, \Box \, N \to L$ in $\mathfrak{C}$;
        
        \item A 2-isomorphism $\tau^f$: \[\begin{tikzcd}
            {M \, \Box \, B \, \Box \, N}
                \arrow[r,"n^M 1"]
                \arrow[d,"1 l^N"']
            & {M \, \Box \, N}
                \arrow[d,"f"]
                \arrow[ld,Rightarrow,shorten <= 10pt,shorten >= 10pt,"\tau^f"]
            \\ {M \, \Box \, N}
                \arrow[r,"f"']
            & {L}
        \end{tikzcd},\]
    \end{enumerate} subject to the following coherence conditions:

    \begin{enumerate}
        \item [a.] We have the pentagon equation
    \end{enumerate}

    \begin{equation} \label{eqn:Balanced1MorAssociativity}
        \begin{tabular}{@{}ccc@{}}
        
        \stringdiagramfigure{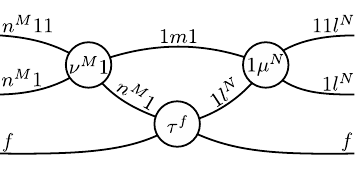} & \adjustbox{valign=c}{$=$} &
        \stringdiagramfigure{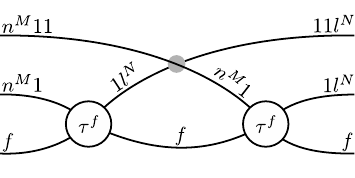}
        
        \end{tabular}
    \end{equation}

    \begin{enumerate}
        \item [] in $\mathbf{Hom}_\mathfrak{C}(M \, \Box \, B \, \Box \, B \, \Box \, N,L)$,
        
        \item [b.] We have the triangle equation
    \end{enumerate}

    \begin{equation} \label{eqn:Balanced1MorUnitality}
        \begin{tabular}{@{}ccc@{}}
        
        \stringdiagramfigure{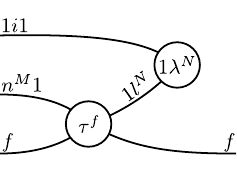} & \adjustbox{valign=c}{$=$} &
        \stringdiagramfigure{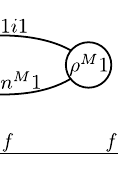}
        
        \end{tabular}
    \end{equation}

    \begin{enumerate}
        \item [] in $\mathbf{Hom}_\mathfrak{C}(M \, \Box \, N,L)$.
    \end{enumerate}
\end{Definition}

\begin{Definition} \label{def:Balanced2Morphism}
    Suppose $f,g:M \, \Box \, N \to L$ are two $B$-balanced 1-morphisms. A $B$-balanced 2-morphism $\xi$ from $(f,\tau^f)$ to $(g,\tau^g)$ consists of a 2-morphism $\xi: f \to g$ in $\mathfrak{C}$ such that the following coherence condition

    \begin{equation} \label{eqn:Balanced2Mor}
        \begin{tabular}{@{}ccc@{}}
        
        \stringdiagramfigure{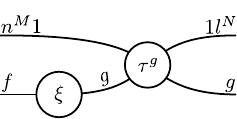} & \adjustbox{valign=c}{$=$} &
        \stringdiagramfigure{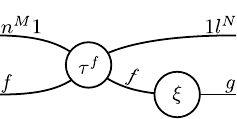}
        
        \end{tabular}
    \end{equation} holds in $\mathbf{Hom}_\mathfrak{C}(M \, \Box \, B \, \Box \, N,L)$.
\end{Definition}

\begin{Proposition}
    There is a category $\mathbf{Bal}_B(M,N;L)$ whose objects are $B$-balanced 1-morphisms and whose morphisms are $B$-balanced 2-morphisms in $\mathfrak{C}$.
\end{Proposition}

\begin{Definition} \label{def:RelativeTensorProduct}
    The relative tensor product of $M$ and $N$ over $B$ consists of:
    \begin{enumerate}
        \item An object $M \, \Box_B \, N$ in $\mathfrak{C}$,
        
        \item A $B$-balanced 1-morphism $(t,\tau^t): M \, \Box \, N \to M \, \Box_B \, N$,
    \end{enumerate} satisfying the following 2-universal property:

    \begin{enumerate}
        \item For any object $L$ in $\mathfrak{C}$ and any $B$-balanced 1-morphism $(f,\tau^f): M \, \Box \, N \to L$, there exists a 1-morphism $\widetilde{f}: M \, \Box_B \, N \to L$ and a 2-isomorphism $\phi:\widetilde{f} \circ_1 t \to f$ in $\mathfrak{C}$;
        
        \item For any 1-morphisms $g,h:M \, \Box_B \, N \to L$ and any $B$-balanced 2-morphism $\zeta: g \circ_1 t \to h \circ_1 t$, there exists a unique 2-morphism $u: g \to h$ in $\mathfrak{C}$ such that $\zeta = u \circ_1 t$.
    \end{enumerate}
\end{Definition}

\begin{Remark}
    The relative tensor product $M \, \Box_B \, N$ can be characterized as the object in $\mathfrak{C}$ that (co-)represents the 2-functor $\mathbf{Bal}_B(M,N;-)$.
\end{Remark}

\begin{Remark}
    We can also view the relative tensor product $M \, \Box_B \, N$ as the 2-colimit\footnote{This is the special type of 2-colimit that is also denoted as \textit{pseudo colimit}, i.e. all 2-cells in its 2-universal property is invertible.} over the diagram \[\begin{tikzcd}
        {M \, \Box \, B \, \Box \, B \, \Box \, N} 
            \arrow[r,shift left= 10pt,"n 1 1"]
            \arrow[r,shift right= 10pt,"1 1 l"]
            \arrow[r,"1 m 1"]
        & {M \, \Box \, B \, \Box \, N}
        \arrow[r,shift left= 5pt,"n 1"]
        \arrow[r,shift right= 5pt,"1 l"]
        & {M \, \Box \, N}
    \end{tikzcd}.\]
\end{Remark}

\begin{Example}
    In $\mathfrak{C} = \mathbf{2Vect}$, the above notions reproduce the usual notions of balanced functor and relative Deligne tensor product \cite{ENO09,DSPS14}.
\end{Example}

\begin{Proposition} \label{prop:AssociativityAndUnitalityOfRelativeTensorProduct}
    Let $A$, $B$, $C$, $D$ be algebras in $\mathfrak{C}$. Suppose relevant relative tensor products exist in the following statements.
    \begin{enumerate}
        \item \textbf{Associativity.} For any $(A,B)$-bimodule $M$, $(B,C)$-bimodule $N$, and $(C,D)$-bimodule $P$, there is a canonical $(A,D)$-bilinear equivalence
        \[
            \pmb{\alpha}_{M,N,P} \colon (M \, \Box_B \, N) \, \Box_C \, P \xrightarrow{\;\sim\;} M \, \Box_B \, (N \, \Box_C \, P),
        \]
        natural in $M$, $N$, and $P$.

        \item \textbf{Unitality.} For any $(A,B)$-bimodule $M$, there are canonical $(A,B)$-bilinear equivalences
        \[
            \pmb{l}_M \colon A \, \Box_A \, M \xrightarrow{\;\sim\;} M \quad \text{and} \quad \pmb{r}_M \colon M \, \Box_B \, B \xrightarrow{\;\sim\;} M,
        \]
        natural in $M$, where $A$ and $B$ are viewed as $(A,A)$- and $(B,B)$-bimodules via their respective multiplications.
    \end{enumerate}
\end{Proposition}

\begin{proof}
    See \cite[3.2.5]{D8} and \cite[3.2.7]{D8}.
\end{proof}

Following Décoppet \cite{D4,D7,D8}, we have the following collection of results.

\begin{Theorem}
    Let $\mathfrak{C}$ be a fusion 2-category.
    \begin{enumerate}
        \item Relative tensor products over separable algebras always exist in $\mathfrak{C}$;
        
        \item For any separable $A$ in $\mathfrak{C}$, it induces a 2-functor \[- \Box_A -: \mathbf{Mod}_\mathfrak{C}(A) \boxtimes \mathbf{Lmod}_\mathfrak{C}(A) \to \mathfrak{C}; \]
        
        \item Relative tensor products are compatible with redundant module structures, i.e. for any separable algebras $A,B,C$ in $\mathfrak{C}$, the above 2-functor gets promoted to \[- \Box_B -: \mathbf{Bimod}_\mathfrak{C}(A,B) \boxtimes \mathbf{Bimod}_\mathfrak{C}(B,C) \to \mathbf{Bimod}_\mathfrak{C}(A,C); \]
        
        \item Moreover, relative tensor products are associative and unital, so that they together form a Morita 3-category\footnote{In general, even if we start with a semistrict monoidal 2-category, the Morita 3-category we constructed is not semistrict; universal properties of relative tensor products induce coherence data, including associator $\pmb{\alpha}$, left unitor $\pmb{l}$, right unitor $\pmb{r}$, see \cite[Section 3.2]{D8}} where \begin{itemize}
            \item Objects are separable algebras in $\mathfrak{C}$,
            
            \item 1-morphisms are bimodules with composition given by relative tensor products,
            
            \item 2-morphisms are bimodule 1-morphisms,
            
            \item 3-morphisms are bimodule 2-morphisms;
        \end{itemize}

        \item This Morita 3-category is equivalent to $\mathbf{Lmod}(\mathfrak{C})$, the 3-category of finite semisimple module 2-categories over $\mathfrak{C}$;
        
        \item Finally, this Morita 3-category is fully dualizable; in particular, every 1-morphism has a left and a right adjoint.
    \end{enumerate}
\end{Theorem}

\begin{Example}
    When $\mathfrak{C} = \mathbf{2Vect}$, the above Morita 3-category has been constructed in \cite{DSPS13}. From a higher condensation point of view \cite{GJF}, this Morita 3-category is the candidate for $\mathbf{3Vect}$. For a general fusion 2-category $\mathfrak{C}$, its Morita 3-category $\mathbf{Lmod}(\mathfrak{C})$ is equivalent to its condensation completion $\Sigma \mathfrak{C} = \mathbf{Kar}(\mathrm{B} \mathfrak{C})$.
\end{Example}

\section{Frobenius Algebras in Monoidal 2-Categories}
Let $\mathfrak{C}$ be a semistrict monoidal 2-category. Then its opposite 2-category $\mathfrak{C}^{1op}$ equipped with the same monoidal product is also a semistrict monoidal 2-category.

\subsection{Coalgebras}

\begin{Definition} \label{def:Coalgebra}
    A coalgebra in $\mathfrak{C}$ is an algebra in $\mathfrak{C}^{1op}$. Equivalently, it consists of:
    \begin{enumerate}
        \item An object $C$ in $\mathfrak{C}$;
        
        \item Two 1-morphisms, comultiplication $\Delta: C \to C \, \Box \, C$ and counit $d:C \to \mathbf{\mathbf{I}}$;
        
        \item Three 2-isomorphisms, coassociator $\alpha$, left counitor $\lambda$ and right counitor $\rho$: \[\begin{tikzcd}
            {C \, \Box \, C \, \Box \, C} 
            & {C \, \Box \, C}
                \arrow[l, "\Delta 1"']
                \arrow[ld,Rightarrow,shorten <= 10pt,shorten >= 10pt,"\alpha"]
            \\ {C \, \Box \, C}
                \arrow[u,"1 \Delta"]
            & {C}
                \arrow[l, "\Delta"]
                \arrow[u, "\Delta"']
        \end{tikzcd}, \begin{tikzcd}
            {\mathbf{I} \, \Box \, C}
                \arrow[d,equal]
            & {C \, \Box \, C}
                \arrow[l, "d 1"']
                \arrow[r,"1 d"]
                \arrow[dl,Rightarrow,shorten <= 10pt,shorten >= 10pt,"\lambda"']
                \arrow[dr,Rightarrow,shorten <= 10pt,shorten >= 10pt,"\rho"]
            & {C \, \Box \, \mathbf{I}} 
                \arrow[d,equal]
            \\ {C}
                \arrow[r, equal]
            & {C}
                \arrow[u, "\Delta"']
            & {C}
                \arrow[l,equal]
        \end{tikzcd},\]
    \end{enumerate} subject to the following coherence conditions:

    \begin{enumerate}
        \item [a.] We have the pentagon equation
    \end{enumerate}

    \begin{equation} \label{eqn:CoalgebraAssociativity}
        \begin{tabular}{@{}ccc@{}}
        
        \stringdiagramfigure{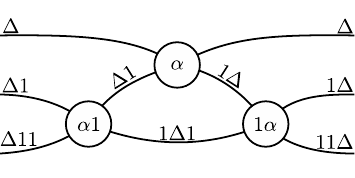} & \adjustbox{valign=c}{$=$} &
        \stringdiagramfigure{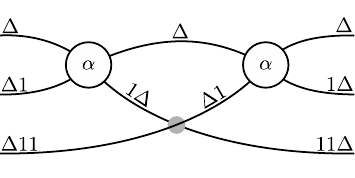}
        
        \end{tabular}
    \end{equation}

    \begin{enumerate}
        \item [] in $\mathbf{Hom}_\mathfrak{C}(C,C \, \Box \, C \, \Box \, C \, \Box \, C)$,
        
        \item [b.] We have the triangle equation
    \end{enumerate}

    \begin{equation} \label{eqn:CoalgebraUnitality}
        \begin{tabular}{@{}ccc@{}}
        
        \stringdiagramfigure{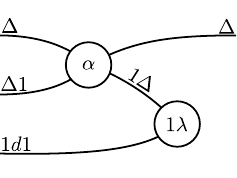} & \adjustbox{valign=c}{$=$} &
        \stringdiagramfigure{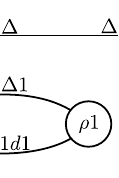}
        
        \end{tabular}
    \end{equation}

    \begin{enumerate}
        \item [] in $\mathbf{Hom}_\mathfrak{C}(C,C \, \Box \, C)$.
    \end{enumerate}
\end{Definition}

\begin{Remark} \label{rmk:Coalgebra}
    A coalgebra in monoidal 2-category $\mathfrak{C}$ is equivalent to an algebra either in monoidal 2-category $\mathfrak{C}^{1op}$ or $\mathfrak{C}^{1op,2op}$. Both cases turn out to be the same because all 2-cells are invertible in the above definition of coalgebras. If we weaken these data to get an (op)lax version of coalgebras in monoidal 2-categories then they will correspond to different (op)lax version of algebras in $\mathfrak{C}^{1op}$ or $\mathfrak{C}^{1op,2op}$, respectively.
\end{Remark}

\begin{Lemma}
    Let $A$ be an algebra in a monoidal 2-category $\mathfrak{C}$ such that the underlying object of $A$ admits a left or right dual. Then its dual object is equipped with a coalgebra structure.
\end{Lemma}

\begin{proof}
    This follows immediately from Remark \ref{rmk:Coalgebra} and the fact that taking duals is a monoidal 2-functor \cite[Appendix A]{DX}.
\end{proof}

\subsection{Frobenius Algebras}

\begin{Definition} \label{def:FrobeniusAlgebra}
    A \textit{Frobenius algebra} in $\mathfrak{C}$ consists of:
    \begin{enumerate}
        \item An object $A$ in $\mathfrak{C}$;

        \item An algebra structure $(A,m,i,\alpha,\lambda,\rho)$ on $A$;

        \item A coalgebra structure $(A,\Delta,d,\widetilde{\alpha},\widetilde{\lambda},\widetilde{\rho})$ on $A$;

        \item The comultiplication $\Delta$ is equipped with an $(A,A)$-bilinear structure with left Frobeniusator $\psi^l$ and right Frobeniusator $\psi^r$: \[\begin{tikzcd}
            {A \, \Box \, A}
                \arrow[r,"1 \Delta"]
                \arrow[d,"m"']
            & {A \, \Box \, A \, \Box \, A}
                \arrow[d,"m 1"]
                \arrow[ld,Rightarrow,shorten <= 15pt,shorten >= 15pt,"\psi^l"]
            \\ {A}
                \arrow[r,"\Delta"']
            & {A \, \Box \, A}
        \end{tikzcd}, \begin{tikzcd}
            {A \, \Box \, A}
                \arrow[r,"\Delta 1"]
                \arrow[d,"m"']
            & {A \, \Box \, A \, \Box \, A}
                \arrow[d,"1 m"]
                \arrow[ld,Rightarrow,shorten <= 15pt,shorten >= 15pt,"\psi^r"]
            \\ {A}
                \arrow[r,"\Delta"']
            & {A \, \Box \, A}
        \end{tikzcd},\]
    \end{enumerate} satisfying the following coherence conditions:

    \begin{enumerate}
        \item [a.] Left Frobeniusator $\psi^l$ witnesses that $\Delta$ is left $A$-linear, i.e. we have
    \end{enumerate}

    \begin{equation} \label{eqn:LeftFrobAssociativity}
        \begin{tabular}{@{}ccc@{}}

        \stringdiagramfigure{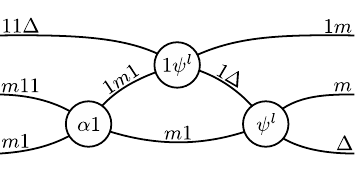} & \adjustbox{valign=c}{$=$} &
        \stringdiagramfigure{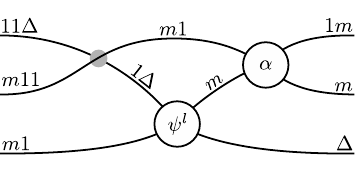}

        \end{tabular}
    \end{equation}

    \begin{enumerate}
        \item [] in $\mathbf{Hom}_\mathfrak{C}(A \, \Box \, A \, \Box \, A,A \, \Box \, A)$, and
    \end{enumerate}

    \begin{equation} \label{eqn:LeftFrobUnitality}
        \begin{tabular}{@{}ccc@{}}

        \stringdiagramfigure{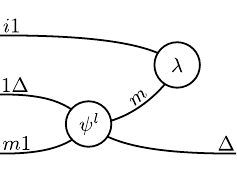} & \adjustbox{valign=c}{$=$} &
        \stringdiagramfigure{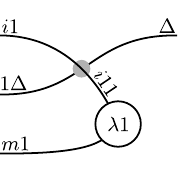}

        \end{tabular}
    \end{equation}

    \begin{enumerate}
        \item [] in $\mathbf{Hom}_\mathfrak{C}(A,A \, \Box \, A)$,

        \item [b.] Right Frobeniusator $\psi^r$ witnesses that $\Delta$ is right $A$-linear, i.e. we have
    \end{enumerate}

    \begin{equation} \label{eqn:RightFrobAssociativity}
        \begin{tabular}{@{}ccc@{}}

        \stringdiagramfigure{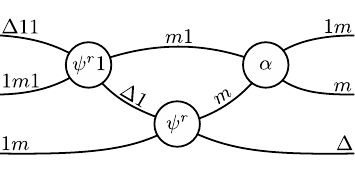} & \adjustbox{valign=c}{$=$} &
        \stringdiagramfigure{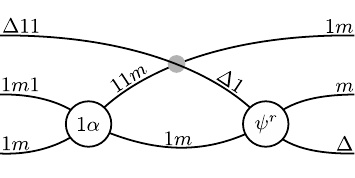}

        \end{tabular}
    \end{equation}

    \begin{enumerate}
        \item [] in $\mathbf{Hom}_\mathfrak{C}(A \, \Box \, A \, \Box \, A,A \, \Box \, A)$, and
    \end{enumerate}

    \begin{equation} \label{eqn:RightFrobUnitality}
        \begin{tabular}{@{}ccc@{}}

        \stringdiagramfigure{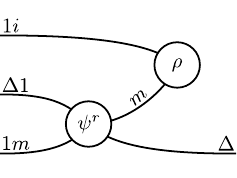} & \adjustbox{valign=c}{$=$} &
        \stringdiagramfigure{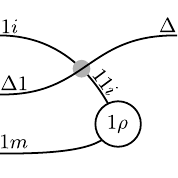}

        \end{tabular}
    \end{equation}

    \begin{enumerate}
        \item [] in $\mathbf{Hom}_\mathfrak{C}(A,A \, \Box \, A)$,

        \item [c.] $\Delta$ is $(A,A)$-bilinear, i.e. we have
    \end{enumerate}

    \begin{equation} \label{eqn:FrobBimod1Mor}
        \begin{tabular}{@{}ccc@{}}

        \stringdiagramfigure{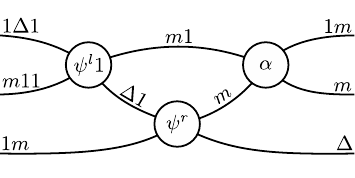} & \adjustbox{valign=c}{$=$} &
        \stringdiagramfigure{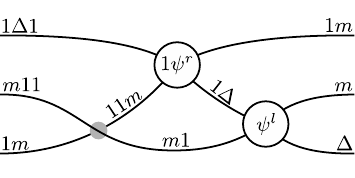}

        \end{tabular}
    \end{equation}

    \begin{enumerate}
        \item [] in $\mathbf{Hom}_\mathfrak{C}(A \, \Box \, A \, \Box \, A,A \, \Box \, A)$.
    \end{enumerate}
\end{Definition}

\begin{Notation}
    When there is ambiguity in distinguishing algebra and coalgebra structures in string diagram calculus, we add superscripts to Frobenius algebra data. For coalgebras, we use $\alpha_{\mathrm{coalg}}$, $\lambda_{\mathrm{coalg}}$, $\rho_{\mathrm{coalg}}$ to denote their coassociators and left and right counitors, respectively.
\end{Notation}

Here we record the additional comodule coherence conditions that the Frobeniusators impose on the multiplication $m$, which will be used in the constructions below.

\begin{Remark} \label{rmk:FrobeniusComoduleConditions}
    Let $A$ be a Frobenius algebra in $\mathfrak{C}$. The Frobeniusators $\psi^l$ and $\psi^r$ also equip the multiplication $m$ with an $(A,A)$-comodule 1-morphism structure. Concretely, the following coherence conditions hold:

    \begin{enumerate}
        \item [a.] We have the pentagon equation
    \end{enumerate}

    \begin{equation} \label{eqn:LeftFrobCoAssociativity}
        \begin{tabular}{@{}ccc@{}}

        \stringdiagramfigure{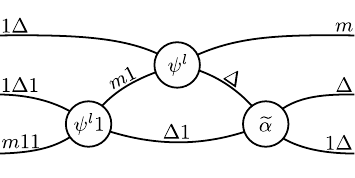} & \adjustbox{valign=c}{$=$} &
        \stringdiagramfigure{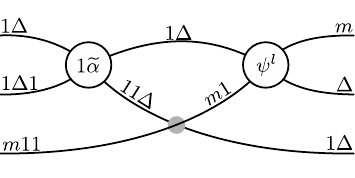}

        \end{tabular}
    \end{equation}

    \begin{enumerate}
        \item [] in $\mathbf{Hom}_\mathfrak{C}(A \, \Box \, A,A \, \Box \, A \, \Box \, A)$,

        \item [b.] We have the triangle equation
    \end{enumerate}

    \begin{equation} \label{eqn:LeftFrobCoUnitality}
        \begin{tabular}{@{}ccc@{}}

        \stringdiagramfigure{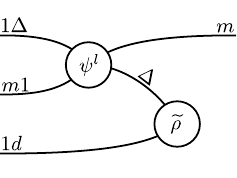} & \adjustbox{valign=c}{$=$} &
        \stringdiagramfigure{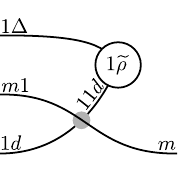}

        \end{tabular}
    \end{equation}

    \begin{enumerate}
        \item [] in $\mathbf{Hom}_\mathfrak{C}(A \, \Box \, A,A)$,

        \item [c.] We have the pentagon equation
    \end{enumerate}

    \begin{equation} \label{eqn:RightFrobCoAssociativity}
        \begin{tabular}{@{}ccc@{}}

        \stringdiagramfigure{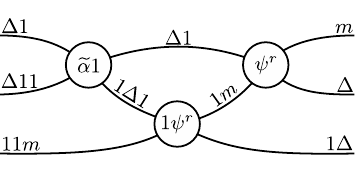} & \adjustbox{valign=c}{$=$} &
        \stringdiagramfigure{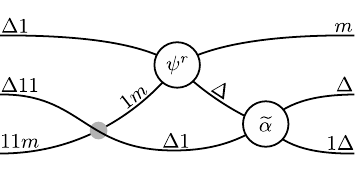}

        \end{tabular}
    \end{equation}

    \begin{enumerate}
        \item [] in $\mathbf{Hom}_\mathfrak{C}(A \, \Box \, A,A \, \Box \, A \, \Box \, A)$,

        \item [d.] We have the triangle equation
    \end{enumerate}

    \begin{equation} \label{eqn:RightFrobCoUnitality}
        \begin{tabular}{@{}ccc@{}}

        \stringdiagramfigure{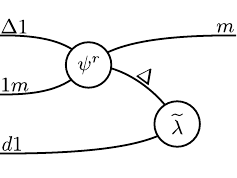} & \adjustbox{valign=c}{$=$} &
        \stringdiagramfigure{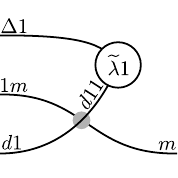}

        \end{tabular}
    \end{equation}

    \begin{enumerate}
        \item [] in $\mathbf{Hom}_\mathfrak{C}(A \, \Box \, A,A)$,

        \item [e.] We have the hexagon equation
    \end{enumerate}

    \begin{equation} \label{eqn:FrobBicomod1Mor}
        \begin{tabular}{@{}ccc@{}}

        \stringdiagramfigure{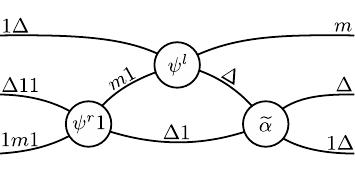} & \adjustbox{valign=c}{$=$} &
        \stringdiagramfigure{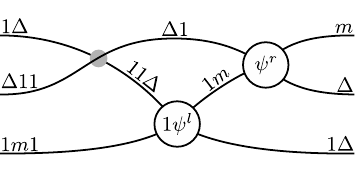}

        \end{tabular}
    \end{equation}

    \begin{enumerate}
        \item [] in $\mathbf{Hom}_\mathfrak{C}(A \, \Box \, A,A \, \Box \, A \, \Box \, A)$.
    \end{enumerate}
\end{Remark}

\subsection{Coherences for Frobenius Algebras}
In the 1-categorical setting, it is well-known that the collection of Frobenius algebra data and axioms can be reduced to a smaller collection, for examples, forgetting the comultiplication $\Delta$ or the counit $d$, see \cite{Kock}. In this section, we present a categorified version of this reduction, which shall be convenient for constructing Frobenius algebras in examples.

\begin{Construction}[Reconstruction of the comultiplication] \label{cstr:ComultiplicationReconstruction}
    Let $(A,m,i,\alpha,\lambda,\rho)$ be an algebra in $\mathfrak{C}$. Then there exists a unique extension to a Frobenius algebra when the counit $d \colon A \to \mathbf{I}$ is provided, together with an extension of the evaluation $e := d \circ m \colon A \, \Box \, A \to \mathbf{I}$ to a coherent self-duality datum on $A$ in $\mathfrak{C}$, i.e., there exists a coevaluation 1-morphism $c \colon \mathbf{I} \to A \, \Box \, A$ and 2-isomorphisms 
    \[ \begin{tikzcd}
            {A}
                \arrow[r,"c 1"]
                \arrow[d,equal]
            & { A\, \Box \, A \, \Box \, A}
                \arrow[d,"1 e"]
                \arrow[dl,Rightarrow,shorten <= 10pt,shorten >= 10pt,"E"]
            \\ {A}
                \arrow[r,equal]
            & {A}
        \end{tikzcd}, \begin{tikzcd}
            {A}
                \arrow[r,"1 c"]
                \arrow[d,equal]
            & { A \, \Box \, A \, \Box \, A}
                \arrow[d,"e 1"]
                \arrow[dl,Rightarrow,shorten <= 10pt,shorten >= 10pt,"F"]
            \\ {A}
                \arrow[r,equal]
            & {A}
        \end{tikzcd}\] satisfying the coherence conditions \eqref{eqn:CoherentDual1} and \eqref{eqn:CoherentDual2}.
    
    The comultiplication 1-morphism $\Delta \colon A \to A \, \Box \, A$ is defined as either of the following composites, 
    \[
       (\mathbf{1}_A \, \Box \, m) \circ (c \, \Box \, \mathbf{1}_A) \;\simeq\; \Delta \;\simeq\; (m \, \Box \, \mathbf{1}_A) \circ (\mathbf{1}_A \, \Box \, c),
    \]
    with the following canonical 2-isomorphism between the two composites:

    \begin{center}
        \stringdiagramfigure{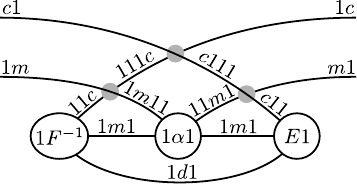} \quad \adjustbox{valign=c}{.}
    \end{center}

    The coassociator, left counitor and right counitor is defined as follows, respectively:

    \begin{center}
        \stringdiagramfigure{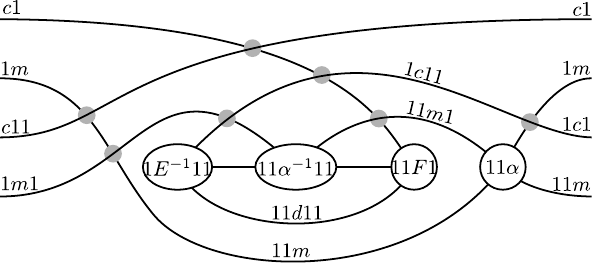}
    \end{center}
    
    \begin{center}
        \stringdiagramfigure{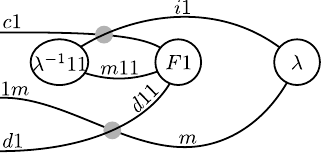} \hspace*{1cm} \stringdiagramfigure{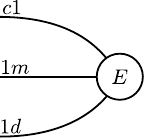} \quad \adjustbox{valign=c}{.}
    \end{center}

    The left and right Frobeniusators are defined as follows, respectively:

    \begin{center}
        \stringdiagramfigure{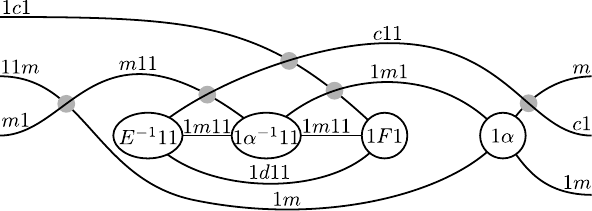}
    \end{center}

    \begin{center} 
        \stringdiagramfigure{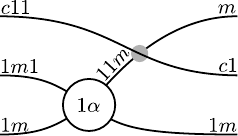} \quad \adjustbox{valign=c}{.}
    \end{center}

    Moreover, the unit $i$ is compatible with the counit $d$ via the coevaluation $c$:

    \begin{center}
        \stringdiagramfigure{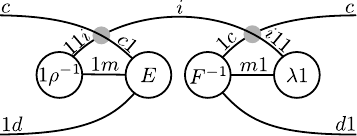} \quad \adjustbox{valign=c}{.}
    \end{center}
\end{Construction}

\begin{Construction}[Reconstruction of the counit] \label{cstr:CounitReconstruction}
    Let $(A,m,i,\alpha,\lambda,\rho)$ be an algebra in $\mathfrak{C}$. Then there exists a unique extension to a Frobenius algebra when the comultiplication $\Delta \colon A \to A \, \Box \, A$ is provided, together with:
    \begin{itemize}[nosep]
        \item Left Frobeniusator $\psi^l$ and right Frobeniusator $\psi^r$, which satisfy the coherence conditions \eqref{eqn:LeftFrobAssociativity}, \eqref{eqn:LeftFrobUnitality}, \eqref{eqn:RightFrobAssociativity}, \eqref{eqn:RightFrobUnitality}, \eqref{eqn:FrobBimod1Mor};
        
        \item An extension of the coevaluation $c := \Delta \circ i \colon \mathbf{I} \to A \, \Box \, A$ to a coherent self-duality datum on $A$ in $\mathfrak{C}$, i.e., there exists an evaluation 1-morphism $e \colon A \, \Box \, A \to \mathbf{I}$ and 2-isomorphisms
        \[ \begin{tikzcd}
            {A}
                \arrow[r,"c 1"]
                \arrow[d,equal]
            & { A\, \Box \, A \, \Box \, A}
                \arrow[d,"1 e"]
                \arrow[dl,Rightarrow,shorten <= 10pt,shorten >= 10pt,"E"]
            \\ {A}
                \arrow[r,equal]
            & {A}
        \end{tikzcd}, \begin{tikzcd}
            {A}
                \arrow[r,"1 c"]
                \arrow[d,equal]
            & { A \, \Box \, A \, \Box \, A}
                \arrow[d,"e 1"]
                \arrow[dl,Rightarrow,shorten <= 10pt,shorten >= 10pt,"F"]
            \\ {A}
                \arrow[r,equal]
            & {A}
        \end{tikzcd}\] satisfying the coherence conditions \eqref{eqn:CoherentDual1} and \eqref{eqn:CoherentDual2};

        \item An $A$-balanced 1-morphism structure on the evaluation $e$, i.e. a 2-isomorphism
        \[\begin{tikzcd}
            {A \, \Box \, A \, \Box \, A}
                \arrow[r,"m 1"]
                \arrow[d,"1 m"']
            & {A \, \Box \, A}
                \arrow[d,"e"]
                \arrow[ld,Rightarrow,shorten <= 15pt,shorten >= 15pt,"\tau"]
            \\ {A \, \Box \, A}
                \arrow[r,"e"']
            & {\mathbf{I}}
        \end{tikzcd}\] satisfying the coherence conditions \eqref{eqn:Balanced1MorAssociativity} and \eqref{eqn:Balanced1MorUnitality}.
    \end{itemize}  
    
    The counit $d \colon A \to \mathbf{I}$ is defined as either of the following composites,
    \[e \circ (i \, \Box \, \mathbf{1}_A) \simeq d \simeq e \circ (\mathbf{1}_A \, \Box \, i),\]
    with the following canonical 2-isomorphism between the two composites:

    \begin{center}
        \stringdiagramfigure{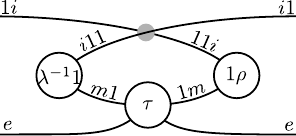} \quad \adjustbox{valign=c}{.}
    \end{center}

    Moreover, the comultiplication $\Delta$ is compatible with the evaluation $e$ via the following 2-isomorphism:

    \begin{center}
        \stringdiagramfigure{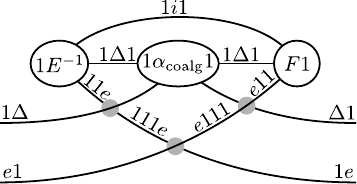} \quad \adjustbox{valign=c}{,}
    \end{center}
    which is further compatible with the multiplication $m$ via the following 2-isomorphism:
    \begin{center}
        \stringdiagramfigure{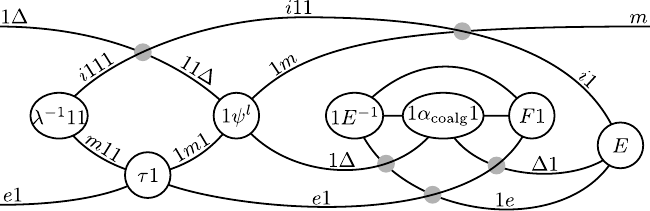} \quad \adjustbox{valign=c}{.}
    \end{center}

    The left and right counitors are defined as follows, respectively:

    \begin{center}
        \stringdiagramfigure{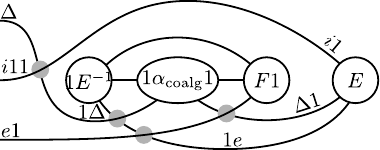}
    \end{center}

    \begin{center}
        \stringdiagramfigure{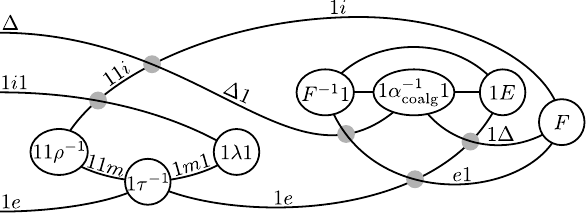} \quad \adjustbox{valign=c}{.}
    \end{center}

\end{Construction}

\subsection{Rigid Algebras}

\begin{Definition} \label{def:RigidAlgebra}
    A \textit{rigid algebra} in $\mathfrak{C}$ consists of:
    \begin{enumerate}
        \item An underlying algebra $(A,m,i,\alpha,\lambda,\rho)$ in $\mathfrak{C}$;

        \item A 1-morphism $\Delta: A \to A \, \Box \, A$ which is an $(A,A)$-bilinear right adjoint to $m$;

        \item More explicitly, one has 2-morphisms, unit $\eta$ and counit $\epsilon$: \[\begin{tikzcd}
            {A}
                \arrow[r,"\Delta"]
                \arrow[d,equal]
            & {A \, \Box \, A}
                \arrow[d, "m"]
                \arrow[r,equal]
                \arrow[ld,Rightarrow,shorten <= 10pt,shorten >= 10pt,"\epsilon"]
            & {A \, \Box \, A}
                \arrow[d,equal]
                \arrow[ld,Rightarrow,shorten <= 10pt,shorten >= 10pt,"\eta"]
            \\ {A}
                \arrow[r,equal]
            & {A}
                \arrow[r,"\Delta"']
            & {A \, \Box \, A}
        \end{tikzcd},\]
    \end{enumerate} satisfying the following coherence conditions:

    \begin{enumerate}
        \item [a.] We have the zigzag equations in $\mathbf{Hom}_\mathfrak{C}(A,A \, \Box \, A)$:
    \end{enumerate}

    \begin{equation} \label{eqn:Zigzag1}
        \begin{tabular}{@{}cccc@{}}

        \adjustbox{valign=c}{$\mathbf{1}_\Delta$} & \adjustbox{valign=c}{$=$} &
        \stringdiagramfigure{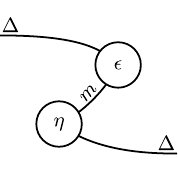} & \adjustbox{valign=c}{,}

        \end{tabular}
    \end{equation}

    \begin{enumerate}
        \item [] and in $\mathbf{Hom}_\mathfrak{C}(A \, \Box \, A,A)$:
    \end{enumerate}

    \begin{equation} \label{eqn:Zigzag2}
        \begin{tabular}{@{}cccc@{}}

        \adjustbox{valign=c}{$\mathbf{1}_m$} & \adjustbox{valign=c}{$=$} &
        \stringdiagramfigure{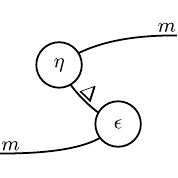} & \adjustbox{valign=c}{,}

        \end{tabular}
    \end{equation}

    \begin{enumerate}
        \item [b.] Counit $\epsilon$ is $(A,A)$-bilinear, i.e. we have
    \end{enumerate}

    \begin{equation} \label{eqn:CounitLeftMod2Mor}
        \begin{tabular}{@{}ccc@{}}

        \stringdiagramfigure{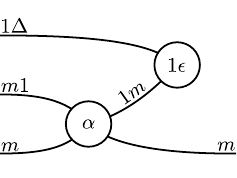} & \adjustbox{valign=c}{$=$} &
        \stringdiagramfigure{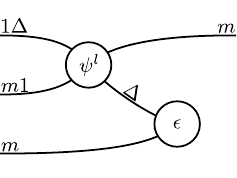}

        \end{tabular}
    \end{equation}

    \begin{enumerate}
        \item [] in $\mathbf{Hom}_\mathfrak{C}(A \, \Box \, A,A)$, and
    \end{enumerate}

    \begin{equation} \label{eqn:CounitRightMod2Mor}
        \begin{tabular}{@{}ccc@{}}

        \stringdiagramfigure{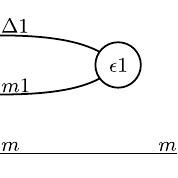} & \adjustbox{valign=c}{$=$} &
        \stringdiagramfigure{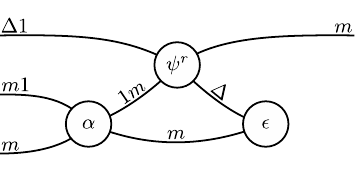}

        \end{tabular}
    \end{equation}

    \begin{enumerate}
        \item [] in $\mathbf{Hom}_\mathfrak{C}(A \, \Box \, A,A)$,

        \item [c.] Unit $\eta$ is $(A,A)$-bilinear, i.e. we have
    \end{enumerate}

    \begin{equation} \label{eqn:UnitLeftMod2Mor}
        \begin{tabular}{@{}ccc@{}}

        \stringdiagramfigure{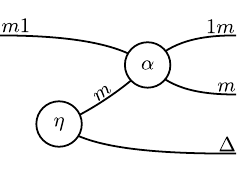} & \adjustbox{valign=c}{$=$} &
        \stringdiagramfigure{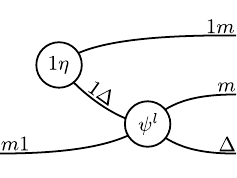}

        \end{tabular}
    \end{equation}

    \begin{enumerate}
        \item [] in $\mathbf{Hom}_\mathfrak{C}(A \, \Box \, A \, \Box \, A,A \, \Box \, A)$, and
    \end{enumerate}

    \begin{equation} \label{eqn:UnitRightMod2Mor}
        \begin{tabular}{@{}ccc@{}}

        \stringdiagramfigure{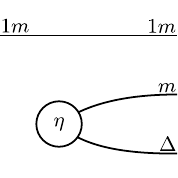} & \adjustbox{valign=c}{$=$} &
        \stringdiagramfigure{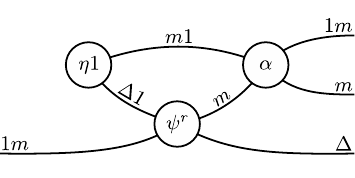}

        \end{tabular}
    \end{equation}

    \begin{enumerate}
        \item [] in $\mathbf{Hom}_\mathfrak{C}(A \, \Box \, A \, \Box \, A,A \, \Box \, A)$.
    \end{enumerate}
\end{Definition}

\begin{Remark}
    When the context is clear, we omit the labels for units and counits of adjoint 1-morphisms in 2-categories. In string diagram calculus, these are simply represented by a cap or a cup, with their zigzag equations corresponding to the straightening of strings.
\end{Remark}

\begin{Lemma} \label{lem:RigidIsFrobenius}
    A rigid algebra $A$ in $\mathfrak{C}$ is a Frobenius algebra, with the same underlying algebra structure and with the coalgebra structure $(A, m^R, i^R, \alpha^R, \lambda^R, \rho^R)$ given by the right adjoint of the algebra structure.
\end{Lemma}

\begin{proof}
    This follows immediately from the definition of rigid algebras and functoriality of taking right adjoints in 2-categories.
\end{proof}

We record the additional comodule structure inherited from the adjunction.

\begin{Lemma}
    Suppose $A$ is a rigid algebra in $\mathfrak{C}$. Then $\Delta:A \to A \, \Box \, A$ is coassociative, i.e.\ the right adjoint of the associator $\alpha$ provides a coassociator $\widetilde{\alpha}$ satisfying equation~(\ref{eqn:CoalgebraAssociativity}).
\end{Lemma}

\begin{Lemma}
    Suppose $A$ is a rigid algebra in $\mathfrak{C}$. Then the comultiplication $\Delta$ with coassociator $\widetilde{\alpha}$ satisfies equations~(\ref{eqn:LeftFrobCoAssociativity}), (\ref{eqn:RightFrobCoAssociativity}), and (\ref{eqn:FrobBicomod1Mor}).
\end{Lemma}

\begin{Remark}
    Suppose $A$ is a rigid algebra in $\mathfrak{C}$. Then the adjunction between $\Delta$ and $m$ also carries an $(A,A)$-comodule structure with Frobeniusators $\psi^l$ and $\psi^r$, giving rise to the additional coherence conditions~(\ref{eqn:CounitLeftComod2Mor})--(\ref{eqn:UnitRightComod2Mor}) below.

    \begin{enumerate}
        \item [a.] We have
    \end{enumerate}

    \begin{equation} \label{eqn:CounitLeftComod2Mor}
        \begin{tabular}{@{}ccc@{}}

        \stringdiagramfigure{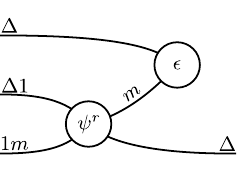} & \adjustbox{valign=c}{$=$} &
        \stringdiagramfigure{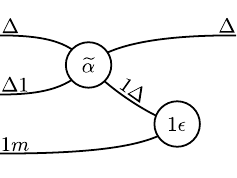}

        \end{tabular}
    \end{equation}

    \begin{enumerate}
        \item [] in $\mathbf{Hom}_\mathfrak{C}(A,A \, \Box \, A)$,

        \item [b.] We have
    \end{enumerate}

    \begin{equation} \label{eqn:CounitRightComod2Mor}
        \begin{tabular}{@{}ccc@{}}

        \stringdiagramfigure{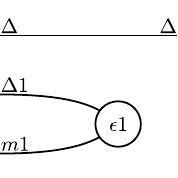} & \adjustbox{valign=c}{$=$} &
        \stringdiagramfigure{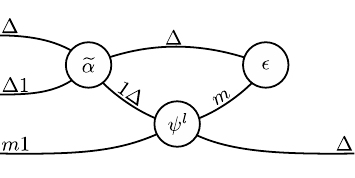}

        \end{tabular}
    \end{equation}

    \begin{enumerate}
        \item [] in $\mathbf{Hom}_\mathfrak{C}(A,A \, \Box \, A)$,

        \item [c.] We have
    \end{enumerate}

    \begin{equation} \label{eqn:UnitLeftComod2Mor}
        \begin{tabular}{@{}ccc@{}}

        \stringdiagramfigure{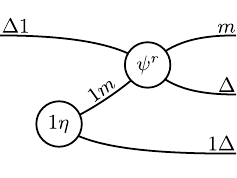} & \adjustbox{valign=c}{$=$} &
        \stringdiagramfigure{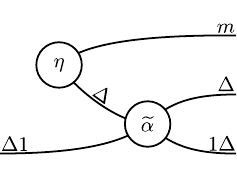}

        \end{tabular}
    \end{equation}

    \begin{enumerate}
        \item [] in $\mathbf{Hom}_\mathfrak{C}(A \, \Box \, A,A \, \Box \, A \, \Box \, A)$,

        \item [d.] We have
    \end{enumerate}

    \begin{equation} \label{eqn:UnitRightComod2Mor}
        \begin{tabular}{@{}ccc@{}}

        \stringdiagramfigure{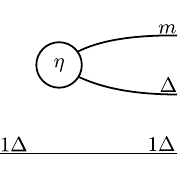} & \adjustbox{valign=c}{$=$} &
        \stringdiagramfigure{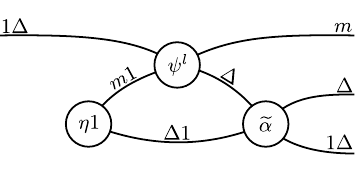}

        \end{tabular}
    \end{equation}

    \begin{enumerate}
        \item [] in $\mathbf{Hom}_\mathfrak{C}(A \, \Box \, A,A \, \Box \, A \, \Box \, A)$.
    \end{enumerate}
\end{Remark}

\begin{Remark}
    A closely related construction appears in Chen et al. \cite[§6]{CFHPS}. Working in a Gray-monoid $\mathfrak{C}$ in which all 1-morphisms admit right adjoints, they consider a rigid algebra $A$ and equip the category $|A| = \mathbf{Hom}_\mathfrak{C}(\mathbf{I},A)$ with a monoidal structure. For each object $a \in |A|$ they build explicit left and right duals $a^\vee$ and ${}^\vee a$ out of the adjunction $a \dashv a^R$ together with the Frobeniator data, and verify the zig-zag identities by an explicit pasting diagram, thereby exhibiting $|A|$ as a multifusion category (their Construction~6.4 and Proposition~6.5). Their algebra hierarchy---algebra, rigid, separable---coincides with the one we use in \S\ref{sec:Algebras}, and, as noted in \cite{CFHPS}, the Frobeniator $\kappa$ is determined by the associator and the adjunction data (an observation attributed there to Reutter), matching our Lemma~\ref{lem:RigidIsFrobenius}. In the unitary ($C^*$ or $H^*$-algebra) setting, \cite{CFHPS} further checks that these duals assemble into a unitary dual functor on $|A|$.
\end{Remark}

\subsection{Special Frobenius Algebras}

\begin{Definition} \label{def:SpecialFrobeniusAlgebra}
    A \textit{special Frobenius algebra} in $\mathfrak{C}$ consists of:
    \begin{enumerate}
        \item An underlying Frobenius algebra $A$ in $\mathfrak{C}$;

        \item An $(A,A)$-bilinear retract $\mathbf{1}_A \xrightarrow{\gamma} m \circ \Delta \xrightarrow{\epsilon} \mathbf{1}_A$,
    \end{enumerate} namely, the following equations hold:

    \begin{enumerate}
        \item [a.] $\gamma$ is a section of $\epsilon$, i.e. in $\mathbf{Hom}_\mathfrak{C}(A,A)$ we have
    \end{enumerate}

    \begin{equation} \label{eqn:CounitSection}
        \begin{tabular}{@{}cccc@{}}

        \adjustbox{valign=c}{$\mathbf{1}_{\mathbf{1}_A}$} & \adjustbox{valign=c}{$=$} &
        \stringdiagramfigure{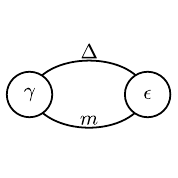} & \adjustbox{valign=c}{,}

        \end{tabular}
    \end{equation}

    \begin{enumerate}
        \item [b.] $\gamma$ is $(A,A)$-bilinear, i.e. we have
    \end{enumerate}

    \begin{equation} \label{eqn:SectionLeftMod2Mor}
        \begin{tabular}{@{}ccc@{}}

        \stringdiagramfigure{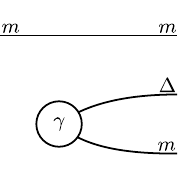} & \adjustbox{valign=c}{$=$} &
        \stringdiagramfigure{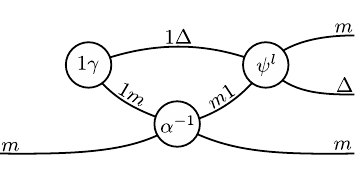}

        \end{tabular}
    \end{equation}

    \begin{enumerate}
        \item [] in $\mathbf{Hom}_\mathfrak{C}(A \, \Box \, A,A)$, and
    \end{enumerate}

    \begin{equation} \label{eqn:SectionRightMod2Mor}
        \begin{tabular}{@{}ccc@{}}

        \stringdiagramfigure{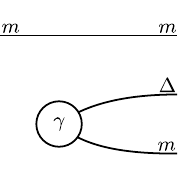} & \adjustbox{valign=c}{$=$} &
        \stringdiagramfigure{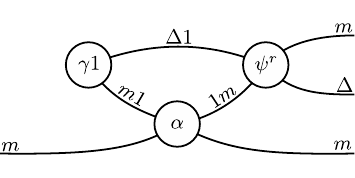}

        \end{tabular}
    \end{equation}

    \begin{enumerate}
        \item [] in $\mathbf{Hom}_\mathfrak{C}(A \, \Box \, A,A)$;

        \item [c.] $\epsilon$ is $(A,A)$-bilinear, i.e. equations~(\ref{eqn:CounitLeftMod2Mor}) and~(\ref{eqn:CounitRightMod2Mor}) hold.
    \end{enumerate}
\end{Definition}

\begin{Definition} \label{def:SeparableAlgebra}
    A \textit{separable algebra} in $\mathfrak{C}$ consists of:
    \begin{enumerate}
        \item An underlying rigid algebra $(A,m,i,\Delta,\alpha,\lambda,\rho,\eta,\epsilon)$ in $\mathfrak{C}$;

        \item An extra special Frobenius algebra structure on $A$, i.e. an $(A,A)$-bilinear section $\gamma: \mathbf{1}_A \to m \circ \Delta$.
    \end{enumerate}
\end{Definition}

\subsection{Comparisons} \label{subsec:Comparisons}

We collect several examples and comparison results relating the algebra hierarchy above. These clarify the relationship between the Frobenius, rigid, special Frobenius, and separable conditions.

\begin{Example} \label{exmp:Rigid_algebra_in_2Vect_is_separable}
    In $\mathfrak{C} = \mathbf{2Vect}$, rigid algebras correspond to multifusion 1-categories, while separable algebras correspond to separable multifusion 1-categories. Moreover, \cite{DSPS13} demonstrates that every multifusion 1-category is separable over a field of characteristic zero.
\end{Example}

Décoppet has generalized this observation to rigid algebras in an arbitrary multifusion 2-category.

\begin{Proposition} \label{prop:Rigid_algebra_in_fusion_2-category_is_separable}
    Over a field $\Bbbk$ of characteristic zero, every rigid algebra in a multifusion 2-category is separable.
\end{Proposition}

\begin{proof}
    In \cite[3.2.4]{D7}, Décoppet provides an equivalent characterization of separability as the non-vanishing of the dimension of the algebra. A later result \cite[3.3.3]{D7} shows this holds for rigid algebras in a connected fusion 2-category. Finally, \cite[4.2.2]{D9} shows that every fusion 2-category is Morita equivalent to a connected fusion 2-category, so the result holds for all multifusion 2-categories.
\end{proof}

\begin{Proposition} \label{prop:Frobenius_algebra_in_2Vect_is_rigid}
    In $\mathfrak{C} = \mathbf{2Vect}$, every special Frobenius algebra is rigid.
\end{Proposition}

\begin{proof}
    Let $\mathcal{C}$ be a finite semisimple 1-category. A special Frobenius monoidal structure on $\mathcal{C}$ consists of:
    \begin{itemize}
        \item A monoidal structure $\otimes \colon \mathcal{C} \boxtimes \mathcal{C} \to \mathcal{C}$ and unit $\mathbf{I} \colon \mathbf{Vect} \to \mathcal{C}$, with associativity and unitality coherence data (which we may assume strict by the coherence theorem);

        \item A comonoidal structure $\Delta \colon \mathcal{C} \to \mathcal{C} \boxtimes \mathcal{C}$ and counit $\epsilon \colon \mathcal{C} \to \mathbf{Vect}$, with coassociativity and counitality coherence data (also assumed strict);

        \item A $(\mathcal{C},\mathcal{C})$-bimodule functor structure on $\Delta$, i.e.\ natural isomorphisms
        \begin{equation*}
            \psi^l \colon (\otimes \boxtimes \mathrm{Id}) \circ (\mathrm{Id} \boxtimes \Delta) \to \Delta \circ \otimes
            \quad \text{and} \quad
            \psi^r \colon (\mathrm{Id} \boxtimes \otimes) \circ (\Delta \boxtimes \mathrm{Id}) \to \Delta \circ \otimes;
        \end{equation*}

        \item Two $(\mathcal{C},\mathcal{C})$-bimodule natural transformations $\mathrm{Id}_\mathcal{C} \xrightarrow{\gamma} \otimes \circ \Delta \xrightarrow{\kappa} \mathrm{Id}_\mathcal{C}$ with $\kappa \circ_2 \gamma = \mathrm{Id}_{\mathrm{Id}_\mathcal{C}}$.
    \end{itemize}

    Since $\mathcal{C}$ is finite semisimple, we decompose $\Delta(\mathbf{I}) \simeq \bigoplus_{i,j = 0}^n \omega^{ij} \, x_i \boxtimes x_j$, where $\mathbf{I} =: x_0, x_1, \dots, x_n$ are the simple objects of $\mathcal{C}$ and $\omega^{ij} \in \mathbb{Z}_{\geq 0}$. The Frobenius condition (the $(\mathcal{C},\mathcal{C})$-bilinear structure on $\Delta$) implies that $(\omega^{ij})$ is the inverse of the matrix $(\epsilon(x_i \otimes x_j))$, hence a permutation matrix. Let $\tau$ be the permutation such that $\omega^{ij} = 1$ if and only if $j = \tau(i)$.

    This implies the counit is co-represented by the simple object $K := x_{\tau(0)} \simeq x_{\tau^{-1}(0)}$, with $\epsilon \simeq \mathbf{Hom}_\mathcal{C}(K,-)$.

    The special Frobenius condition implies $\mathbf{Nat}_{\mathcal{C}|\mathcal{C}}(\mathrm{Id}_\mathcal{C}, \otimes \circ \Delta) \neq 0$. Consider the \emph{Casimir object} $\Omega := \otimes \circ \Delta(\mathbf{I}) \simeq \bigoplus_{i=0}^n x_i \otimes x_{\tau(i)}$. The $(\mathcal{C},\mathcal{C})$-bimodule structure on $\Delta$ gives
    \begin{equation} \label{eqn:CasimirLeftDistributivity}
        \otimes \circ \Delta(y) \xrightarrow[\;\sim\;]{\otimes \circ_1 (\psi^l_{y,\mathbf{I}})^{-1}} y \otimes \Omega,
    \end{equation}
    and
    \begin{equation} \label{eqn:CasimirRightDistributivity}
        \otimes \circ \Delta(y) \xrightarrow[\;\sim\;]{\otimes \circ_1 (\psi^r_{\mathbf{I},y})^{-1}} \Omega \otimes y.
    \end{equation}

    The projection $p\colon \Omega \to x_0 \otimes x_{\tau(0)} \simeq K$ induces an isomorphism
    \[\mathbf{Nat}_{\mathcal{C}|\mathcal{C}}(\mathrm{Id}_\mathcal{C}, \otimes \circ \Delta) \xrightarrow{\;\sim\;} \mathbf{Hom}_\mathcal{C}(\mathbf{I},K),\]
    since the $(\mathcal{C},\mathcal{C})$-bimodule structure determines a natural transformation from its value at $\mathbf{I}$, and $\epsilon(x_i \otimes x_{\tau(i)}) \simeq \mathbf{Hom}_\mathcal{C}(K, x_i \otimes x_{\tau(i)})$ is one-dimensional for each $i$. Since $\mathbf{Nat}_{\mathcal{C}|\mathcal{C}}(\mathrm{Id}_\mathcal{C}, \otimes \circ \Delta) \neq 0$, we have $\mathbf{Hom}_\mathcal{C}(\mathbf{I},K) \neq 0$. As $\mathbf{I}$ and $K$ are both simple, Schur's lemma gives $\mathbf{I} \simeq K$.

    To prove rigidity, it suffices to show every simple object has duals. Using the coassociativity of $\Delta$ and the decomposition of $\Delta(\mathbf{I})$, one checks that
    \[ y \xrightarrow{\;\sim\;} \bigoplus_{i=0}^n \epsilon(y \otimes x_{\tau^{-1}(i)}) \boxtimes x_i \quad \text{and} \quad y \xrightarrow{\;\sim\;} \bigoplus_{i=0}^n x_i \boxtimes \epsilon(x_{\tau(i)} \otimes y). \]
    Since $\mathbf{I} \simeq K$, the spaces $\epsilon(x_{\tau^{-1}(i)} \otimes x_i) \simeq \mathbf{Hom}_\mathcal{C}(\mathbf{I}, x_{\tau^{-1}(i)} \otimes x_i)$ are one-dimensional, so one can choose non-zero coevaluation maps. By Schur's lemma, $\dim \mathbf{Hom}_\mathcal{C}(a,b) = \dim \mathbf{Hom}_\mathcal{C}(b,a)$ for simple $a,b$, so $x_{\tau(i)} \simeq x_{\tau^{-1}(i)}$ non-canonically, and evaluation maps satisfying the zigzag identities can be chosen.
\end{proof}

\begin{Remark} \label{rmk:CasimirIsCentralCommutativeAlgebra}
    In the setting of Proposition~\ref{prop:Frobenius_algebra_in_2Vect_is_rigid}, the Casimir object $\Omega = \otimes \circ \Delta(\mathbf{I})$ carries a canonical special Frobenius algebra structure in $\mathcal{C}$. Moreover, the isomorphisms~(\ref{eqn:CasimirLeftDistributivity}) and~(\ref{eqn:CasimirRightDistributivity}) equip $\Omega$ with a canonical half-braiding, so that it lifts to a commutative algebra in the Drinfeld center $\mathcal{Z}_1(\mathcal{C})$.
\end{Remark}

\begin{Remark}
    A general Frobenius algebra in $\mathbf{2Vect}$ is not necessarily rigid. For instance, take a finite semisimple 1-category $\mathcal{C}$ with two simple objects $\mathbf{I}$ and $K$ such that $K \otimes K = 0$. Then $\mathcal{C}$ admits a Frobenius structure with $\Delta(\mathbf{I}) = K \boxtimes \mathbf{I} \oplus \mathbf{I} \boxtimes K$, $\Delta(K) = K \boxtimes K$, $\epsilon(\mathbf{I}) = 0$, $\epsilon(K) = \Bbbk$, but $K$ has no dual, so $\mathcal{C}$ is not rigid.
\end{Remark}

\begin{Remark} \label{rmk:SpecialFrobeniusAlgebraMightNotBeRigid}
    In a general monoidal 2-category, special Frobenius algebras might not be rigid. The author is not aware of a counterexample; one could plausibly be constructed in a monoidal 2-category where the existence of adjoints for 1-morphisms is severely restricted.
\end{Remark}

Whether every special Frobenius algebra in a fusion 2-category is rigid remains an open question. Proposition~\ref{prop:Frobenius_algebra_in_2Vect_is_rigid} establishes this in the case $\mathfrak{C} = \mathbf{2Vect}$ by an explicit combinatorial argument, but that argument does not immediately generalize. We will expand on these questions in the future work.

% A proof in full generality would need to resolve at least one of the following sub-questions:
% \begin{itemize}
%     \item \textbf{Morita invariance.} If $A$ is a special Frobenius algebra in a fusion 2-category $\mathfrak{C}$, and $A$ is Morita equivalent to another algebra $B$, does $B$ carry a canonical special Frobenius structure?

%     \item \textbf{Morita equivalence of ambient 2-categories.} If every special Frobenius algebra in a fusion 2-category $\mathfrak{C}$ is rigid, and $\mathfrak{C} \simeq \mathfrak{C}'$ are Morita equivalent fusion 2-categories, does the same hold in $\mathfrak{C}'$? (This is motivated by Décoppet's result \cite[4.2.2]{D9} that every fusion 2-category is Morita equivalent to a connected one.)

%     \item \textbf{Forgetful functor.} For a braided fusion 1-category $\mathcal{A}$, the forgetful 2-functor $\mathbf{Mod}(\mathcal{A}) \to \mathbf{2Vect}$ is not strongly monoidal. Can it nonetheless transfer special Frobenius structures, and can rigidity in $\mathbf{2Vect}$ be lifted to rigidity in $\mathbf{Mod}(\mathcal{A})$?
% \end{itemize}

\section{Main Result}

At the end of Preliminaries, we recall that there exists a fully dualizable Morita 3-category of separable algebras within every fusion 2-category $\mathfrak{C}$. In particular, for separable algebras $A,B$ in $\mathfrak{C}$, any $(A,B)$-bimodule $M$ has both a left and a right dual. However, the proofs in the literature \cite{D8} do not provide explicit constructions of these duals. We now provide such constructions in a slightly more general setting. In the following, we fix two Frobenius algebras $A$ and $B$ in a semistrict monoidal 2-category $\mathfrak{C}$ (we add superscripts $A$ or $B$ to distinguish their data; see the Notation in \S\ref{sec:Algebras}). For convenience, we assume relative tensor products over $A$ and $B$ exist in $\mathfrak{C}$.

\subsection{Definitions}

\begin{Definition}
    Let $M$ be an $(A,B)$-bimodule in $\mathfrak{C}$. A coherent right dual of $M$ consists of:
    \begin{enumerate}
        \item An object $N$ in $\mathfrak{C}$,
        
        \item A $(B,A)$-bimodule structure on $N$,
        
        \item A $(B,B)$-bilinear 1-morphism $\widetilde{c}:B \to N \, \Box_A \, M$,
        
        \item An $(A,A)$-bilinear 1-morphism $\widetilde{e}:M \, \Box_B \, N \to A$,
        
        \item Two 2-isomorphisms \[\begin{tikzcd}[column sep=30pt]
            {B \, \Box_B \, N}
                \arrow[r,"\widetilde{c} \, \Box_B \, N"]
            & { (N \, \Box_A \, M) \, \Box_B \, N}
                \arrow[dd,"\pmb{\alpha}_{N,M,N}"]
                \arrow[ddl,Rightarrow,shorten <= 30pt,shorten >= 30pt,"\widetilde{E}"]
            \\ {N}
                \arrow[u,"\pmb{l}^{-1}_N"]
            & {}
            \\ {N \, \Box_A \, A}
                \arrow[u,"\pmb{r}_{N}"]
            & {N \, \Box_A \, (M \, \Box_B \, N)}
                \arrow[l,"N \, \Box_A \, \widetilde{e}"]
        \end{tikzcd}, \] \[ \begin{tikzcd}[column sep=30pt]
            {M \, \Box_B \, B}
                \arrow[r,"M \, \Box_B \, \widetilde{c}"]
            & {M \, \Box_B \, (N \, \Box_A \, M)}
                \arrow[dd,"\pmb{\alpha}^{-1}_{M,N,M}"]
                \arrow[ddl,Rightarrow,shorten <= 30pt,shorten >= 30pt,"\widetilde{F}"]
            \\ {M}
            \arrow[u,"\pmb{r}^{-1}_{M}"]
            & {}
            \\ {A \, \Box_A \, M}
                \arrow[u,"\pmb{l}_{M}"]
            & {(M \, \Box_B \, N) \, \Box_A \, M}
                \arrow[l,"\widetilde{e} \, \Box_A \, M"]
        \end{tikzcd},\]
    \end{enumerate} such that 
    \begin{enumerate}
        \item [a.] $\widetilde{E}$ is a $(B,A)$-bilinear 2-isomorphism;
        
        \item [b.] $\widetilde{F}$ is an $(A,B)$-bilinear 2-isomorphism;
        
        \item [c.] $\widetilde{E}$ and $\widetilde{F}$ satisfy the coherence equations (\ref{eqn:CoherentDual1}) and (\ref{eqn:CoherentDual2}) with slight modifications, i.e. we need to insert extra coherence data\footnote{Even when $\mathfrak{C}$ is semistrict, the composition of relative tensor products in $\mathfrak{C}$ are in general associative and unital up to some coherence data provided by the 2-universal properties \cite{D8}.} for the composition of relative tensor products in suitable places.
    \end{enumerate}

    Similarly, one can define a coherent left dual of $M$.
\end{Definition}

\noindent Let us denote the $(A,B)$-bimodule structure on $M$: 

\begin{enumerate}
    \item $M$ has left $A$-module structure $(l,\mu,\lambda)$,
    
    \item $M$ has right $B$-module structure $(n,\nu,\rho)$,
    
    \item $M$ has bimodule associator $\alpha$.
\end{enumerate}

\noindent Fix a coherent right dual $(M^\lor,c,e,E,F)$ of the underlying object of $M$ in $\mathfrak{C}$: \[c: \mathbf{I} \to M^\lor \, \Box \, M,\quad e: M \, \Box \, M^\lor \to \mathbf{I},\] \[ \begin{tikzcd}
            {M^\lor}
                \arrow[r,"c 1"]
                \arrow[d,equal]
            & { M^\lor \, \Box \, M \, \Box \, M^\lor}
                \arrow[d,"1 e"]
                \arrow[dl,Rightarrow,shorten <= 10pt,shorten >= 10pt,"E"]
            \\ {M^\lor}
                \arrow[r,equal]
            & {M^\lor}
        \end{tikzcd}, \begin{tikzcd}
            {M}
                \arrow[r,"1 c"]
                \arrow[d,equal]
            & { M \, \Box \, M^\lor \, \Box \, M}
                \arrow[d,"e 1"]
                \arrow[dl,Rightarrow,shorten <= 10pt,shorten >= 10pt,"F"]
            \\ {M}
                \arrow[r,equal]
            & {M}
        \end{tikzcd}.\]

\subsection{Bimodule Structure on the Dual}
$M^\lor$ can be promoted to a coherent right dual of $(A,B)$-bimodule $M$ with the following data.

\begin{Construction} \label{cstr:RightActionOnDualBimodule}
    We construct the right $A$-action on $M^\lor$ via the composition \[M^\lor \, \Box \, A \xrightarrow{c 1 1} M^\lor \, \Box \, M \, \Box \, M^\lor \, \Box \, A \xrightarrow{1 i^A 1 1 1} M^\lor \, \Box \, A \, \Box \, M \, \Box \, M^\lor \, \Box \, A\] \[\xrightarrow{1 \Delta^A 1 1 1} M^\lor \, \Box \, A \, \Box \, A \, \Box \, M \, \Box \, M^\lor \, \Box \, A \xrightarrow{1 1 l 1 1} M^\lor \, \Box \, A \, \Box \, M \, \Box \, M^\lor \, \Box \, A \] \[\xrightarrow{1 1 e 1} M^\lor \, \Box \, A \, \Box \, A  \xrightarrow{1 m^A} M^\lor \, \Box \, A \xrightarrow{1 d^A} M^\lor, \] with associator 
    
    \begin{equation} \label{eqn:DualBimoduleRightActionAssociator}
        \begin{tabular}{@{}c@{}}

        \stringdiagramfigure{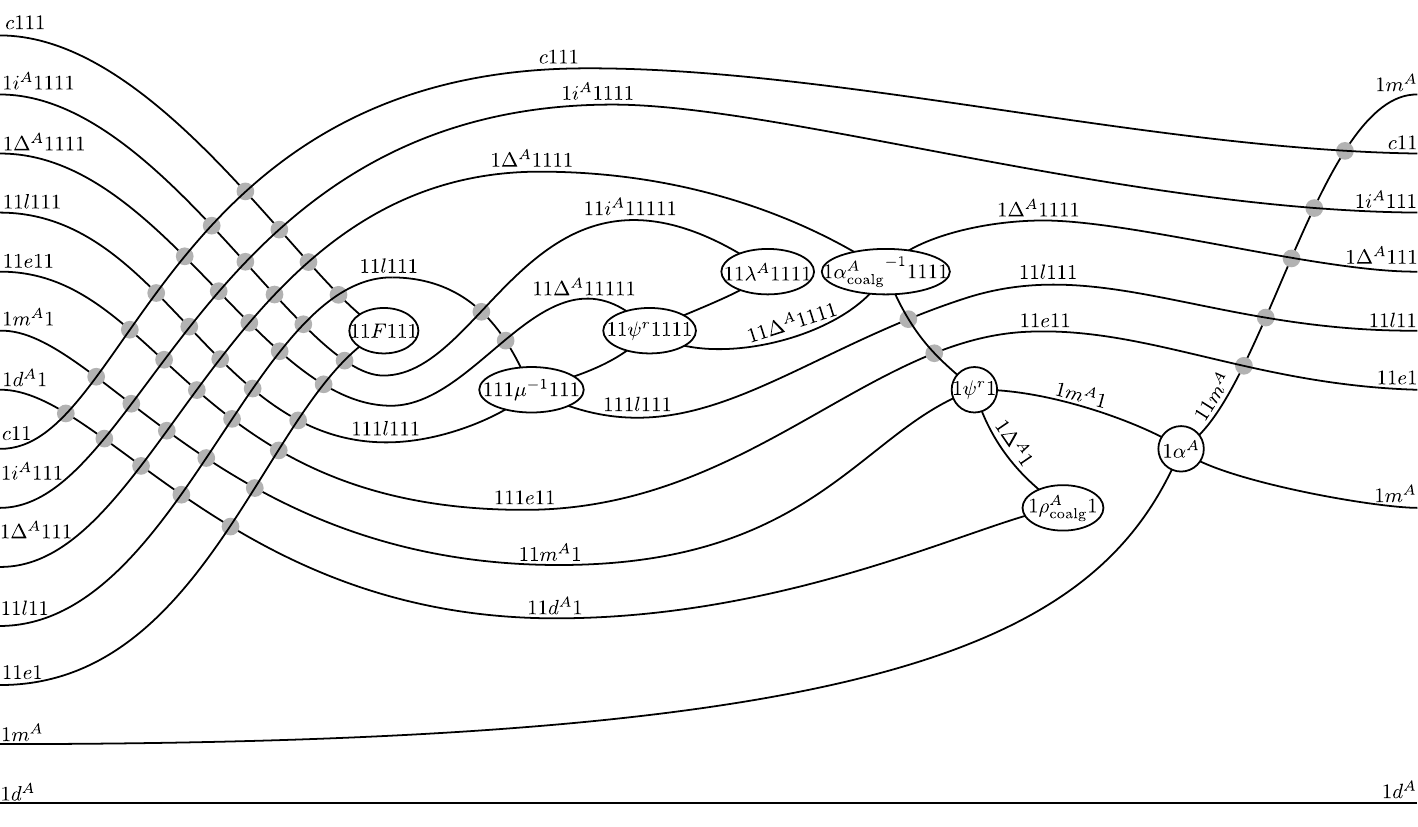}

        \end{tabular}
    \end{equation}
    
    and right unitor

    \begin{equation} \label{eqn:DualBimoduleRightActionUnitor}
        \begin{tabular}{@{}cc@{}}
        
        \stringdiagramfigure{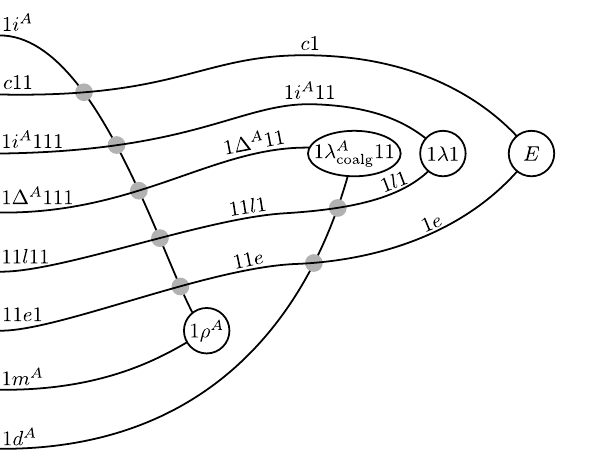} & \adjustbox{valign=c}{.}
        
        \end{tabular}
    \end{equation}
\end{Construction}

\begin{Construction} \label{cstr:LeftActionOnDualBimodule}
    We construct the left $B$-action on $M^\lor$ via the composition \[B \, \Box \, M^\lor \xrightarrow{c 1 1} M^\lor \, \Box \, M \, \Box \, B \, \Box \, M^\lor \xrightarrow{1 n 1} M^\lor \, \Box \, M \, \Box \, M^\lor \xrightarrow{1 e} M^\lor,\] with associator 

    \begin{equation} \label{eqn:DualBimoduleLeftActionAssociator}
        \begin{tabular}{@{}c@{}}

        \stringdiagramfigure{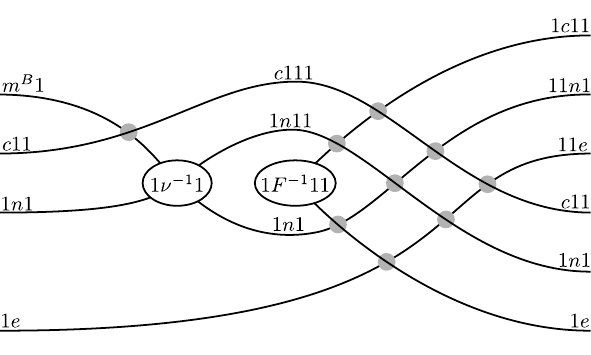}
        
        \end{tabular}
    \end{equation}

    and left unitor

    \begin{equation} \label{eqn:DualBimoduleLeftActionUnitor}
        \begin{tabular}{@{}cc@{}}
        
        \stringdiagramfigure{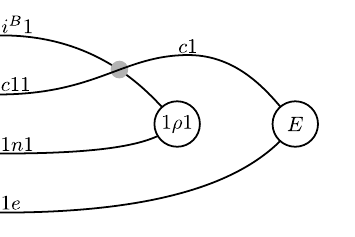} & \adjustbox{valign=c}{.}
        
        \end{tabular}
    \end{equation}
\end{Construction}

\begin{Construction}
    The $(A,B)$-bimodule associator on $M^\lor$ is constructed as follows:

    \begin{equation} \label{eqn:DualBimoduleBimoduleAssociator}
        \begin{tabular}{@{}cc@{}}
        
        \stringdiagramfigure{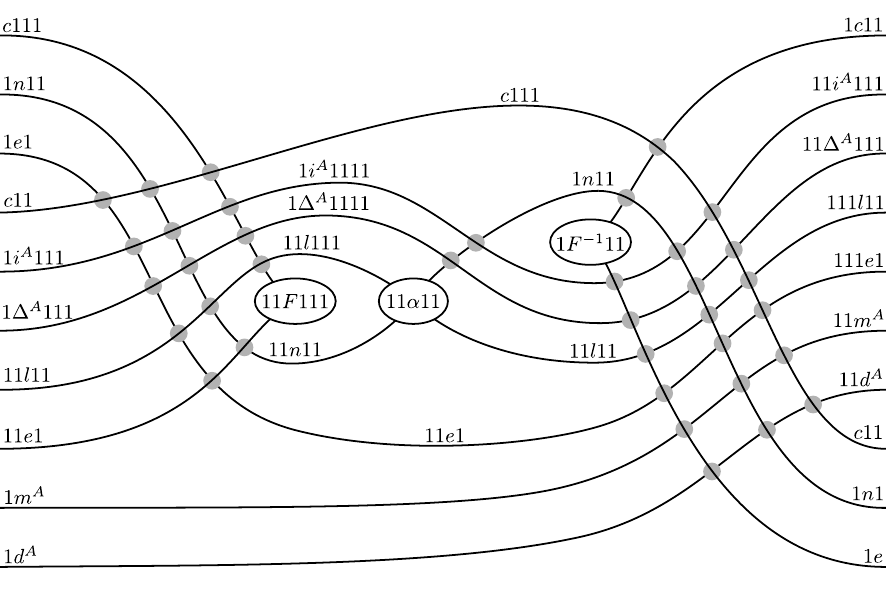} & \adjustbox{valign=c}{.}
        
        \end{tabular}
    \end{equation}
\end{Construction}

\subsection{Duality Coherence 1-Morphisms}

\begin{Construction}
    We construct the unit $\widetilde{c}: B \to M^\lor \, \Box_A \, M$ via composing \[B \xrightarrow{c 1} M^\lor \, \Box \, M \, \Box \, B \xrightarrow{1 n} M^\lor \, \Box \, M \xrightarrow{t} M^\lor \, \Box_A \, M, \] with $t$ the canonical $A$-balanced 1-morphism induced by the universal property of relative tensor product (Definition \ref{def:RelativeTensorProduct}).

    Unit $\widetilde{c}$ also inherits a $(B,B)$-bilinear 1-morphism structure from those on \[B \xrightarrow{c 1} M^\lor \, \Box \, M \, \Box \, B \xrightarrow{1 n} M^\lor \, \Box \, M \] defined via

    \begin{equation} \label{eqn:DualBimoduleUnitLeftAction}
        \begin{tabular}{@{}c@{}}
        
        \stringdiagramfigure{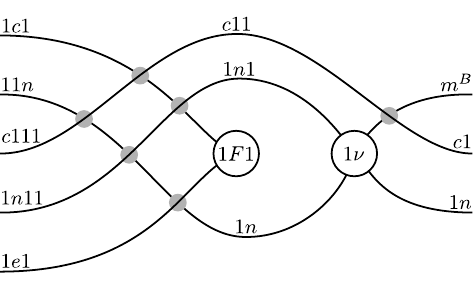}
        
        \end{tabular}
    \end{equation}

    and

    \begin{equation} \label{eqn:DualBimoduleUnitRightAction}
        \begin{tabular}{@{}cc@{}}
        
        \stringdiagramfigure{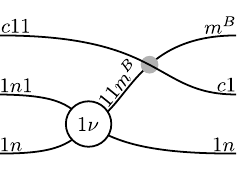} & \adjustbox{valign=c}{.}
        
        \end{tabular}
    \end{equation}
\end{Construction}

\begin{Construction}
    We construct the counit $\widetilde{e}: M \, \Box_B \, M^\lor \to A$ via considering \[M \, \Box \, M^\lor \xrightarrow{i^A 1 1} A \, \Box \, M \, \Box \, M^\lor \xrightarrow{\Delta^A 1 1} A \, \Box \, A \, \Box \, M \, \Box \, M^\lor \xrightarrow{1 l 1} A \, \Box \, M \, \Box \, M^\lor \xrightarrow{1 e} A,\] with the following $B$-balanced structure

    \begin{equation} \label{eqn:DualBimoduleCounitBalanced}
        \begin{tabular}{@{}cc@{}}
        
        \stringdiagramfigure{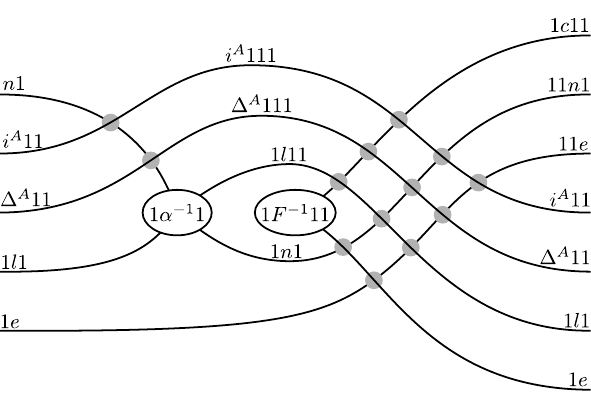} & \adjustbox{valign=c}{.}
        
        \end{tabular}
    \end{equation}
    
    Therefore, by the universal property of relative tensor product (Definition \ref{def:RelativeTensorProduct}), we get the desired $\widetilde{e}$.

    Moreover, counit $\widetilde{e}$ inherits an $(A,A)$-bilinear 1-morphism structure from those on \[M \, \Box \, M^\lor \xrightarrow{i^A 1 1} A \, \Box \, M \, \Box \, M^\lor \xrightarrow{\Delta^A 1 1} A \, \Box \, A \, \Box \, M \, \Box \, M^\lor \xrightarrow{1 l 1} A \, \Box \, M \, \Box \, M^\lor  \xrightarrow{1 e} A,\] defined via

    \begin{equation} \label{eqn:DualBimoduleCounitLeftAction}
        \begin{tabular}{@{}c@{}}

        \stringdiagramfigure{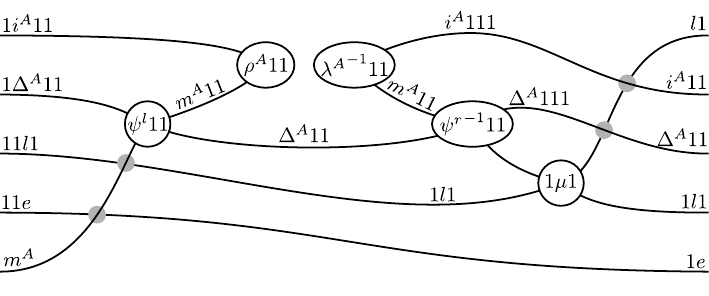}

        \end{tabular}
    \end{equation}

    \noindent and\hfill\null

    \begin{equation} \label{eqn:DualBimoduleCounitRightAction}
        \begin{tabular}{@{}cc@{}}

        \stringdiagramfigure{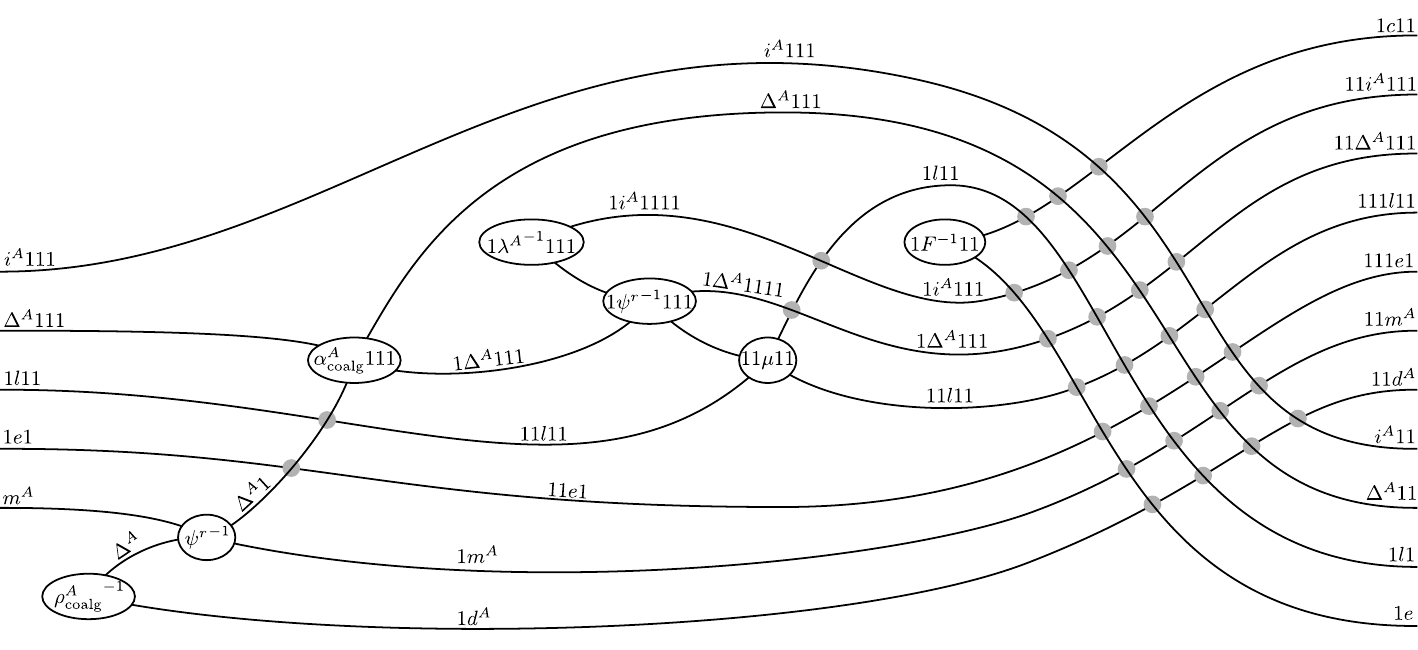} & \adjustbox{valign=c}{.}

        \end{tabular}
    \end{equation}
\end{Construction}

\subsection{Duality Coherence 2-Isomorphisms} 
All data constructed above (the left $B$-action, right $A$-action, bimodule associator on $M^\lor$, and the unit $\widetilde{c}$ and counit $\widetilde{e}$) only require the \emph{Frobenius} algebra structures on $A$ and $B$ (Definition~\ref{def:FrobeniusAlgebra}). However, the zigzag 2-isomorphisms $\widetilde{E}$ and $\widetilde{F}$ require the \emph{special} Frobenius structures (Definition~\ref{def:SpecialFrobeniusAlgebra}). This is the categorification of the analogous 1-categorical phenomenon: to lift duality data for a bimodule over Frobenius algebras, one needs a special Frobenius section on the appropriate side to construct the inverses of the zigzag 2-isomorphisms.

\begin{Construction}
    We construct the 2-isomorphism $\widetilde{E}$ by descending a balanced 2-cell from the absolute tensor products. Consider the following diagram:
    \begin{center}
    \resizebox{\textwidth}{!}{\ensuremath{\begin{tikzcd}[ampersand replacement=\&]
        {M^\vee}
            \arrow[d,"\pmb{l}^{-1}_{M^\vee}"']
        \& {}
        \& {B \, \Box \, M^\vee}
            \arrow[d,"c 1 1"]
            \arrow[ll,"\eqref{cstr:LeftActionOnDualBimodule}"']
            \arrow[lld,"t_{B,M^\vee}"]
        \\
        {B \, \Box_B \, M^\vee}
            \arrow[d,"\widetilde{c} \, \Box_B \, M^\vee"']
        \& {}
        \& {M^\vee \, \Box \, M \, \Box \, B \, \Box \, M^\vee}
            \arrow[d,"1 n 1"]
        \\
        {(M^\vee \, \Box_A \, M) \, \Box_B \, M^\vee}
            \arrow[d,"{\pmb{\alpha}_{M^\vee,M,M^\vee}}"']
        \& {(M^\vee \, \Box_A \, M) \, \Box \, M^\vee}
            \arrow[l,"t"']
        \& {M^\vee \, \Box \, M \, \Box \, M^\vee}
            \arrow[l,"t 1"']
            \arrow[ld,"1 t"]
            \arrow[d,"1 i^A 1 1"]
        \\
        {M^\vee \, \Box_A \, (M \, \Box_B \, M^\vee)}
            \arrow[dd,"M^\vee \, \Box_A \, \widetilde{e}"']
        \& {M^\vee \, \Box \, (M \, \Box_B \, M^\vee)}
            \arrow[l,"t"]
        \& {M^\vee \, \Box \, A \, \Box \, M \, \Box \, M^\vee}
            \arrow[d,"1 \Delta^A 1 1"]
        \\
        {}
        \& {}
        \& {M^\vee \, \Box \, A \, \Box \, A \, \Box \, M \, \Box \, M^\vee}
            \arrow[d,"1 1 l 1"]
        \\
        {M^\vee \, \Box_A \, A}
            \arrow[d,"\pmb{r}_{M^\vee}"']
        \& {}
        \& {M^\vee \, \Box \, A \, \Box \, M \, \Box \, M^\vee}
            \arrow[d,"1 1 e"]
        \\
        {M^\vee}
        \& {}
        \& {M^\vee \, \Box \, A}
            \arrow[ll,"\eqref{cstr:RightActionOnDualBimodule}"]
            \arrow[ull,"t_{M^\vee,A}"']
    \end{tikzcd}}}
    \end{center}
    where the 2-cells are filled by invertible 2-morphisms as follows:
    \begin{itemize}
        \item The two triangles and the middle pentagon are filled by the associativity and unitality constraints $\pmb{\alpha}$, $\pmb{l}$, $\pmb{r}$ for relative tensor products (Proposition \ref{prop:AssociativityAndUnitalityOfRelativeTensorProduct});
        
        \item The upper pentagon is filled by the definition of $\widetilde{c}$ and functoriality of the balanced structure on the relative tensor product $- \, \Box_B \, M^\vee$;
        
        \item The lower polygon is filled by the definition of $\widetilde{e}$ and functoriality of the balanced structure on the relative tensor product $M^\vee \, \Box_A \, -$.
    \end{itemize}

    Continue from the right-bottom sides of the above diagram, we further construct a 2-morphism whose target is the left $B$-action on $M^\vee$ constructed in Construction \ref{cstr:LeftActionOnDualBimodule}:

    \begin{equation} \label{eqn:LiftZigZagE}
        \begin{tabular}{@{}cc@{}}
        \stringdiagramfigure{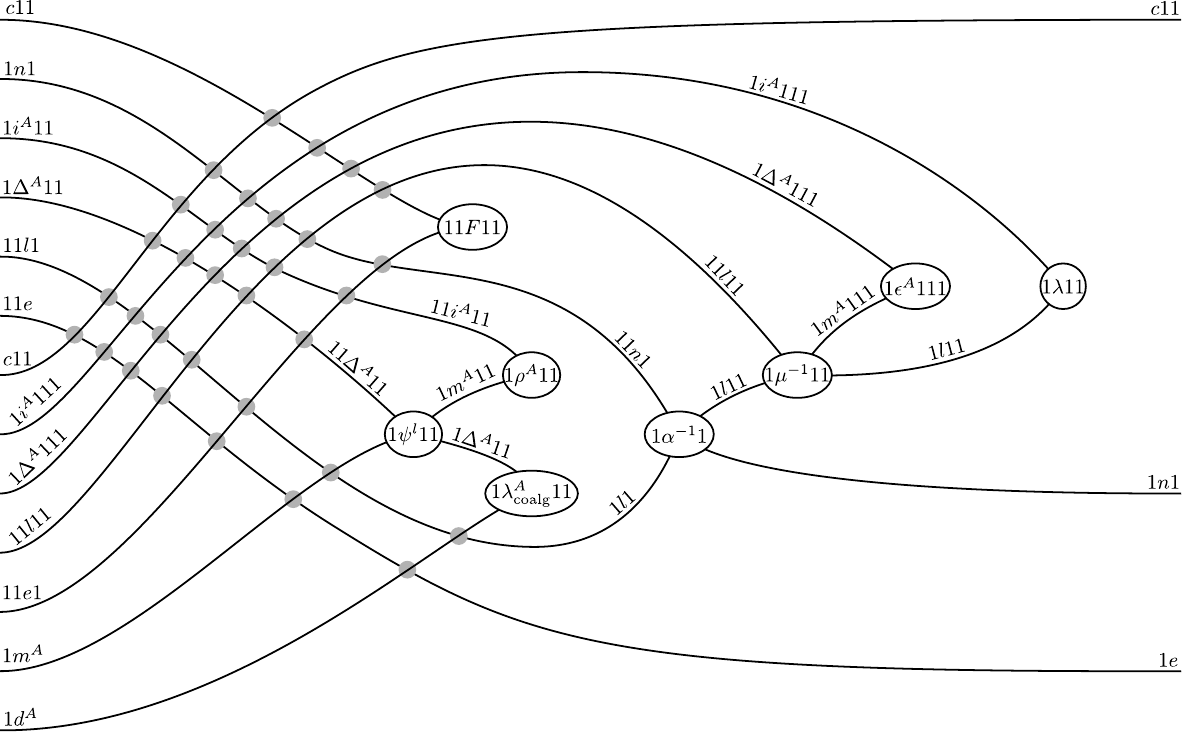} & \adjustbox{valign=c}{$.$}
        \end{tabular}
    \end{equation}

    To construct the inverse $\widetilde{E}^{-1}$, we can simply reverse the 2-cells in the above diagrams \emph{except} the one $\epsilon^A$, which should be replaced by its section $\gamma^A$ (see Definition \ref{def:SpecialFrobeniusAlgebra}). Although the 2-morphisms constructed in this way are not invertible, they are both $B$-balanced (equation \eqref{eqn:Balanced2Mor}) and their descendants become inverses of each other.

\end{Construction}

\begin{Construction}
    Similarly, we construct the 2-isomorphism $\widetilde{F}$ by descending a balanced 2-cell from the absolute tensor products. Consider the following diagram:
    \[\begin{tikzcd}
        {M}
            \arrow[d,"\pmb{r}^{-1}_M"']
        & {}
        & {M \, \Box \, B}
            \arrow[d,"1 c 1"]
            \arrow[ll,"n"']
            \arrow[lld,"t_{M,B}"]
        \\ {M \, \Box_B \, B}
            \arrow[d,"M \, \Box_B \, \widetilde{c}"']
        & {}
        & {M \, \Box \, M^\vee \, \Box \, M \, \Box \, B}
            \arrow[d,"1 1 n"]
        \\ {M \, \Box_B \, (M^\vee \Box_A M)}
            \arrow[d,"\pmb{\alpha}^{-1}_{M,M^\vee,M}"']
        & {M \, \Box \, (M^\vee \Box_A M)}
            \arrow[l,"t"']
        & {M \, \Box \, M^\vee \, \Box \, M}
            \arrow[l,"1t"']
            \arrow[ld,"t 1"]
            \arrow[d,"i^A 1 1 1"]
        \\ {(M \, \Box_B \, M^\vee) \, \Box_A \, M}
            \arrow[dd,"\widetilde{e} \, \Box_A \, M"']
        & {(M \, \Box_B \, M^\vee) \, \Box \, M}
            \arrow[l,"t"]
        & {A \, \Box \, M \, \Box \, M^\vee \, \Box \, M}
            \arrow[d,"\Delta^A 1 1 1"]
        \\ {}
        & {}
        & {A \, \Box \, A \, \Box \, M \, \Box \, M^\vee \, \Box \, M}
            \arrow[d,"1 l 1 1"]
        \\ {A \, \Box_A \, M}
            \arrow[d,"\pmb{l}_M"']
        & {}
        & {A \, \Box \, M \, \Box \, M^\vee \, \Box \, M}
            \arrow[d,"1 e 1"]
        \\ {M}
        & {}
        & {A \, \Box \, M}
            \arrow[ll,"l"]
            \arrow[ull,"t_{A,M}"']
    \end{tikzcd} \]
    where the 2-cells are filled by invertible 2-morphisms as follows:
    \begin{itemize}
        \item The two triangles and the middle pentagon are filled by the associativity and unitality constraints $\pmb{\alpha}$, $\pmb{l}$, $\pmb{r}$ for relative tensor products (Proposition \ref{prop:AssociativityAndUnitalityOfRelativeTensorProduct});
        
        \item The upper pentagon is filled by the definition of $\widetilde{c}$ and functoriality of the balanced structure on the relative tensor product $M \, \Box_B \, -$;
        
        \item The lower polygon is filled by the definition of $\widetilde{e}$ and functoriality of the balanced structure on the relative tensor product $- \, \Box_A \, M$.
    \end{itemize}

    Continue from the right-bottom sides of the above diagram, we further construct a 2-morphism whose target is $n \colon M \, \Box \, B \to M$:
    \begin{equation} \label{eqn:LiftZigZagF}
        \begin{tabular}{@{}cc@{}}
        \stringdiagramfigure{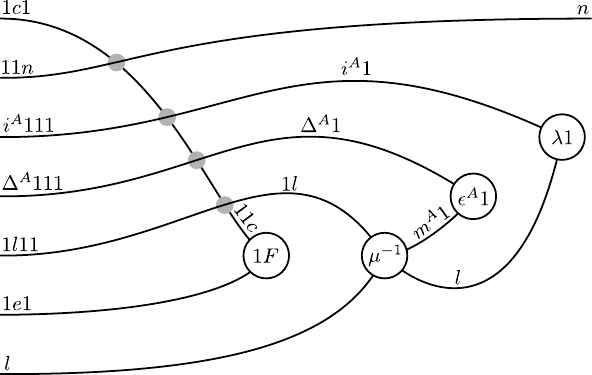} & \adjustbox{valign=c}{$.$}
        \end{tabular}
    \end{equation}

    The inverse $\widetilde{F}^{-1}$ is similarly constructed by reversing the 2-cells in the above diagrams \emph{except} $\epsilon^A$, which is replaced by its section $\gamma^A$.
\end{Construction}

\bibliography{bibliography.bib}

\end{document}